\documentclass{article}

\title{Improving the Convergence Rates for the Kinetic Fokker-Planck Equation by Optimal Control}

\usepackage{algorithm} % e.g. for the algorithm environment
\usepackage{algpseudocode} % e.g. for the algorithmic environment
\usepackage{amsmath} % e.g. for equation* environment
\usepackage{amssymb}
\usepackage{amsthm}
\usepackage{cleveref}
\usepackage{enumerate} % for Roman numbering in enumerations
\usepackage{bbm}
\usepackage[utf8]{inputenc}
\usepackage{mathtools} % e.g. for the command \vcentcolon
\usepackage{subcaption}
\usepackage{xcolor} % e.g. for the command \textcolor

\newcommand{\matlab}{MATLAB\textsuperscript{\textregistered}}

\theoremstyle{plain}
\newtheorem{theorem}{Theorem}

\newtheorem{lemma}[theorem]{Lemma}
\newtheorem{proposition}[theorem]{Proposition}
\newtheorem{remark}[theorem]{Remark}

\newtheorem{innercustomthm}{Assumption}
\newenvironment{customthm}[1]
{\renewcommand\theinnercustomthm{#1}\innercustomthm}
{\endinnercustomthm}

\newcommand{\dd}{\, \mathrm{d}}
\newcommand{\dxdv}{\, \mathrm{d}x\,\mathrm{d}v}
\newcommand{\ddt}{\tfrac{\mathrm{d}}{\mathrm{d}t}}

\newcommand{\tcb}{\color{black}}

\newcommand{\cL}{\mathcal{L}}

\newcommand{\Cc}{C^\infty_0(\mathbb{R}^{2d})}

\begin{document}

% \maketitle
\maketitle

\centerline{\scshape Tobias Breiten}
 \medskip
 {\footnotesize
  \centerline{Institute of Mathematics}
    \centerline{Technische Universit\"at Berlin}
    \centerline{Stra\ss e des 17. Juni 136, 10623 Berlin, Germany}
    \centerline{tobias.breiten@tu-berlin.de}
 }

 \medskip

 \centerline{\scshape Karl Kunisch}
 \medskip
 {\footnotesize
  \centerline{Institute of Mathematics}
    \centerline{University of Graz, Austria}
    \centerline{Radon Institute}
    \centerline{Austrian Academy of Sciences, Linz, Austria}
    \centerline{karl.kunisch@uni-graz.at}
 }

\bigskip

% REQUIRED
\begin{abstract}
The long time behavior and detailed convergence analysis of Langevin equations has received increased attention over the last years. Difficulties arise from a lack of coercivity, usually termed hypocoercivity, of the underlying kinetic Fokker-Planck operator which is a consequence of the partially deterministic nature of a second order stochastic differential equation. In this manuscript, the effect of controlling the confinement potential without altering the original invariant measure is investigated. This leads to an abstract bilinear control system with an unbounded but infinite-time admissible control operator which, by means of an artificial diffusion approach, is shown to possess a unique solution. The compactness of the underlying semigroup is further used to define an infinite-horizon optimal control problem on an appropriately reduced state space. Under smallness assumptions on the initial data, feasibility of and existence of a solution to the optimal control problem are discussed. Numerical results based on a local approximation based on a shifted Riccati equation illustrate the theoretical findings.
\end{abstract}
 
{\em Keywords: kinetic Fokker-Planck, hypocoercivity,  stabilization, optimal control}

% REQUIRED
% \begin{MSCcodes}
% 35Q84, 35Q93, 47N70, 93D15
% % 68Q25, 68R10, 68U05
% \end{MSCcodes}

\section{Introduction}\label{s1}

We consider at first the following uncontrolled Langevin dynamics satisfying the stochastic differential equation in $\mathbb{R}^{2n}$:
\begin{equation}\label{eq:langevin}
\left\{
\begin{aligned}
    \mathrm{d} x_t &=  v_t \mathrm{d}t \\
    \mathrm{d} v_t &= -\nabla G(x_t) \mathrm{d}t - \gamma v_t \mathrm{d}t + \sqrt{\tfrac{2\gamma}{\beta}}\, \mathrm{d}W_t,
  \end{aligned}
  \right.
\end{equation}
where $G(x_t) \ge 0$ is a confinement potential, and $\gamma$ and $\beta$ are positive constants. In the case of thermodynamic ensembles the constant $\beta$ relates to the inverse temperature, and $\gamma$ denotes a friction coefficient, see, e.g., \cite{LelS16,Ris96}.
The evolution of the underlying probability density $\psi$ is known to satisfy the following Kolmogorov forward equation, also called the kinetic Fokker-Planck equation,
\begin{align}\label{eq:kf-fp}
  \ddt \psi = \mathcal{L}^\sharp \psi , \ \ \psi(0) = \psi_0,
\end{align}
where $\psi_0$ denotes the initial state and $\mathcal{L}^\sharp$ the \emph{formal} adjoint in  $L^2$ of the operator
\begin{align}\label{eq:l}
  \mathcal{L}\varphi = v^\top \nabla_x \varphi- \nabla_x G^\top \nabla_v \varphi + \gamma\,( -v^\top   \nabla _v \varphi + \beta^{-1} \Delta_v \varphi).
\end{align}
It is given by
\begin{align}\label{eq:ldagger}
 \mathcal{L}^\sharp \psi = - v^\top \nabla_x \psi +  \nabla_x G^\top \nabla_v \psi + \gamma\,\mathrm{div}_v(v \psi + \beta^{-1} \nabla_v \psi ).
\end{align}

It is well-known that under suitable assumptions on $G$, there exists a canonical invariant measure defined as
\begin{align*}
 \mu (\mathrm{d}x \mathrm{d}v) = Z_\mu^{-1} \mathrm{exp}(- H(x,v) )\, \mathrm{d}x\, \mathrm{d}v,
\end{align*}
where $H(x,v)=\frac{1}{2}\|v\|_{\mathbb{R}^n}^2 + G(x)$ and $Z_\mu=\int_{\mathbb R^{2d}}  e^{-H(x,v)}\, \mathrm{d}x\, \mathrm{d}v$, and that the solutions to \eqref{eq:kf-fp} convergence to $\mu$ at an exponential rate. Since the process \eqref{eq:langevin} is non-reversible, its generator $\mathcal{L}$ is however non-normal and this exponential rate of convergence can be small, see \cite{GroS16} for a detailed discussion on the influence of $\beta$ and $\gamma$ in this regard. Additionally, the factor multiplying the exponential decay, i.e., the transient bound can be large, see, e.g., \cite{GroS16,HelN05,LelS16}.

This motivates the introduction of controls into \eqref{eq:langevin} as it has been discussed in, e.g., \cite{Breetal21,CheGP15,LelNP13}. One of the main challenges in such a control framework is to ensure that the invariant measure of the uncontrolled dynamics \eqref{eq:langevin} is unaltered. In the previous articles, this has been studied for linear feedback forces which enter as an additional summand into the second equation in \eqref{eq:langevin}. Contrary, here we are interested in the effect of modifying the underlying confinement potential $G$ by adding a separable control term $u\alpha$ leading to
\begin{equation}\label{eq:langevincontrol}
\left\{
  \begin{aligned}
    \mathrm{d} x_t &=  v_t \mathrm{d}t \\
    \mathrm{d} v_t &= -\nabla G(x_t)\mathrm{d}t - u_t \nabla\alpha(x_t)  \mathrm{d}t - \gamma v_t \mathrm{d}t + \sqrt{\tfrac{2\gamma}{\beta}}\, \mathrm{d}W_t,
  \end{aligned}
  \right.
\end{equation}
where $\alpha$ denotes a fixed (spatial) control potential and $u_t$ denotes a scalar-valued time-dependent control. The results of this work can be extended to the case of multiple control inputs in a straightforward manner. In the context of particle physics the control is motivated by speeding up the uncontrolled dynamics by means of focusing the intensity of a laser beam of shape $\alpha$ and intensity $u$, see \cite{JonMV15}. Our in interest in the control of \eqref{eq:langevin} also relates to  the fact that it serves as a tool for asymptotic analysis of generalized stochastic gradient descent methods (\cite[Theorem 14]{LiTE19}). Similar to control strategies related to specific learning rates \cite{LiTE15}, we aim at improving the performance of these methods, in this article via the influence of $u\,\alpha$.

We recall  \cite{Ris96} that for $\gamma \to \infty$ the Langevin equation can be approximated by the Smoluchowski equation given by  
\begin{equation*}
\mathrm{d} x_t = -\nabla G(x_t)\mathrm{d}t +  \sqrt{\tfrac{2}{\beta}}\, \mathrm{d}W_t,
\end{equation*}
for which the associated probability density function satisfies a Fokker-Planck equation with a uniformly elliptic generator.
The operators $\mathcal{L}$  and  $\mathcal{L}^\sharp$ on the other hand are elliptic only in the  variable $v$.

The operators $\mathcal{L}$ and $\mathcal{L}^\sharp$ belong to the class of hypoelliptic operators \cite{Hoe67}. Despite its lack of ellipticity w.r.t.~the spatial variable $x$, under suitable assumptions on $G$, the operators $\mathcal{L}$ and $\mathcal{L}^\sharp$ and the associated dynamics asymptotically converge to the unique invariant measure. This property is typically refereed to as \emph{hypocoercivity}, see the seminal article \cite{Vil09}. Asymptotic properties of $\mathcal{L}$ have already been addressed earlier for specific confinement potentials arising in statistical mechanics \cite{EckH00,EckH03}. In particular, the latter works already provide a detailed functional analytic framework of the underlying semigroup and the spectral properties of its infinitesimal generator. Contrary to the case of the Fokker-Planck equation, the spectra of the operators $\mathcal{L}$ and $\mathcal{L}^\sharp$ contain cusps such that the semigroups are no longer analytic. In \cite{HelN05,HerN04} it is however shown that there is a smoothing effect which renders these semigroups compact, a fact that will be exploited later in this manuscript. Generally, the operators $\mathcal{L}$ and $\mathcal{L}^\sharp$, their spectral properties, and associated evolution equations have been the focus of intense research, see for instance \cite{ArnE14,DolMS15,GroS16,GuiM16}, and references given there. The investigation of control techniques for equations involving hypocoercive operators appears to not have been addressed from a PDE perspective yet.

 Let us also observe that  the control enters into \eqref{eq:langevincontrol} into a bilinear structure with the state. We refer to \cite{BalS79} as one of the earliest  papers which analyzed feedback control for bilinear abstract systems and to \cite{AnnB13} for early work on optimal control of the Fokker-Planck equation.

The structure and a brief summary of the paper are as follows. In Section 2 we gather operator theoretic facts for $\mathcal{L}$ which will be relevant in later sections. Section 3 is devoted to the study of the controlled kinetic Fokker-Planck equation and the existence of solutions for the inhomogeneous equation. Since the control which enters to \eqref{eq:langevincontrol} does not influence the subspace which is generated by the invariant measure, an appropriate orthogonal decomposition of the state space is introduced and incorporated into the control problem in Section 4. The optimal stabilization problem is formulated in Section 5. Existence to this problem is verified. This requires stabilizability results which go beyond the stability properties which the uncontrolled system already enjoys. They are obtained by means of the infinite-dimensional Hautus test for stabilizability. For the linearized optimal control problem existence of a solution to the associated Riccati equation is obtained. This allows to speed up of the exponential stability for sufficiently small initial conditions and the representation of feedback controls. In Section 6 numerical results obtained by a spectral collocation technique are presented.\\

\noindent

{\bf{Notation.}} By $C^\infty_0(\mathbb{R}^{2d})$ we denote the set of all functions in $C^\infty(\mathbb{R}^{2d})$ with compact support in $\mathbb{R}^{2d}$. For Hilbert spaces $Y,Z$ the space of bounded linear operators is denoted by $\mathcal{B}(Y,Z)$, if $Y=Z$, we simply write $\mathcal{B}(Z)$. For a linear (unbounded) operator $A$ with domain $\mathcal{D}(A)$, we write $A\colon \mathcal{D}(A)\subset Z\to Z$. We distinguish between the formal adjoint $A^*$ of $A$ in $Z$ and the Hilbert space adjoint $A^\dagger$ of $A$ in $Z$. For a closed, densely defined linear operator $A$ with domain $\mathcal{D}(A)$ in $Z$ we shall also consider $A$ as a bounded linear operator $A\in \mathcal{B}(\mathcal{D}(A),Z)$ where $\mathcal{D}(A)$ is endowed with the graph norm. Its dual $A' \in \mathcal{B}(Z',[\mathcal{D}(A)]')$ is uniquely defined and it is the unique extension of the operator $A^\dagger \in \mathcal{B}(\mathcal{D}(A^\dagger),Z)$ to an element of $\mathcal{B}(Z,[\mathcal{D}(A)]')$. For the associated duality pairing $\langle \cdot ,\cdot \rangle _{\mathcal{D}(A^\dagger),[\mathcal{D}(A^\dagger)]'}$, we simply write $\langle \cdot,\cdot \rangle _{\mathcal{D}}$. In case $Z=L^2(\mathbb R^{2d})$ we denote by $A^\sharp$ the formal adjoint of $A$ in $L^2(\mathbb R^{2d})$ and by $A^\ddagger$ the Hilbert space adjoint of $A$ in $L^2(\mathbb R^{2d})$. Throughout the paper, we will extensively use the weighted (Hilbert) spaces
\begin{equation*}
\begin{aligned}
&Y\!=\!L^2_\mu(\mathbb R^{2d})\! =\! \left\{ y\colon \mathbb R^{2d}\! \to\! \mathbb R \ | \ \mu^{\frac{1}{2}} y \in L^2(\mathbb R^{2d}) \right\},  \| y\| _Y \!= \!\left( \int_{\mathbb R^{2d}} \mu y^2 \, \mathrm{d}x\, \mathrm{d}v\right)^{\frac{1}{2}}, \\[1ex]
&V\!=\!H_\mu^1(\mathbb R^{2d})\!=\!\left\{ y\colon \mathbb R^{2d}\!\to \!\mathbb R\ | \ y\in Y, \nabla y \in Y^{2d} \right\}, \| y\|_V \!= \! \left(\|y\|^2_Y + \|\nabla y \|_{Y^{2d}}^2\right)^{\frac{1}{2}},  \\[1ex]
 & V_v\!=\! H_{\mu,v}^1(\mathbb R^{2d})\!=\!\left\{ y\colon \mathbb R^{2d}\!\to\! \mathbb R\ |  \  y\in Y, \nabla_v y \in Y^d \right\},   \| y\|_{V_v} \!\!=\!\! \left(\|y\|^2_Y \!+ \!\|\nabla_v y \|_{Y^d}^2\right)^{\frac{1}{2}},
 \end{aligned}
 \end{equation*}
where $\nabla_vy=\begin{pmatrix} \tfrac{\partial y}{\partial v_1},\dots, \tfrac{\partial y}{\partial v_1}\end{pmatrix}^\top$.

\section{The Operator form of the kinetic Fokker-Planck equation}\label{s2}

In this short section we recall some developments concerning the operator  $\mathcal{L}$ as generator of a $C_0$ semigroup.  For simplicity of presentation  we set the parameters $\gamma$ and $\beta$ equal to 1 and return to this point when addressing the asymptotic behavior of solutions.
 The  formal adjoint $\cL^\sharp$ in $L^2(\mathbb{R}^{2d})$ of $\cL$ with domain $C^\infty_0(\mathbb{R}^{2d})$ has already been defined.
The formal adjoint $\mathcal{L}^*$ of $\mathcal{L}$ w.r.t.~the measure $\mu$ and  domain $C^\infty_0$ is given by  
\begin{align}\label{eq:l*}
 \mathcal{L}^* \phi = -v^\top \nabla_x \phi + \nabla_x G^\top \nabla_v \phi   -v^\top \nabla_v \phi +  \Delta_v \phi,
\end{align}
see for instance \cite[Section 2.3.2]{LelS16}.
 Here and in the following we frequently use that
\begin{equation}\label{eq:aux_1}
\begin{aligned}
 \nabla_x \mu &= -  \mu \nabla_x G,  \quad  \nabla_v \mu = -  \mu  v.
\end{aligned}
\end{equation}
Throughout the remainder we assume that the confinement potential $G$ satisfies

\begin{customthm}{A1} \label{ass:A1}
The confinement potential satisfies $G\in W^{1,\infty}_{\mathrm{loc}}(\mathbb{R}^d)$, it is \\  bounded from below, and $\int_{\mathbb{R}^d} e^{-G(x)} \,\dd x<\infty.$
\end{customthm}

We further define the isomorphism $
  \mathcal{M}\colon L_\mu^2(\mathbb R^{2d}) \to L^2(\mathbb R^{2d}) , \ \mathcal{M} f = \mu^{\frac{1}{2}}f.
$
A short computation shows that $\mathcal{L}^*=\mathcal{M}^{-2} \mathcal{L}^\sharp \mathcal{M}^2$.
It has been pointed  out, for instance in \cite[I.7.4]{Vil09} that for a probabilistic interpretation it is more appropriate to consider
\begin{align}\label{eq:kf-fpw}
  \ddt \varphi = \mathcal{L}^* \varphi , \ \ \varphi(0) = \varphi_0,
\end{align}
in $Y$, rather than \eqref{eq:kf-fp} in $L^2(\mathbb{R}^{2d})$, and we shall follow this formulation in the sequel. The state variables of these two equations are related by $\psi=\mathcal{M}^2 \varphi$.

Next we address semigroup properties of $\cL$ and $\cL^*$. It is our understanding of the literature that, for quite some time,  it has been taken for granted that the closures of these operators generate $C_0$-semigroups before rigorous proofs were given.

\begin{proposition}\label{prop:adjoint_relation}

Let Assumption \eqref{ass:A1} hold. Then the operators $\overline{\cL}$ and $\overline{\cL^*}$ are maximally dissipative in $Y$ and they generate contraction semigroups $e^{\overline{\cL} t}$ and  $e^{\overline{\cL^*} t}$ on $Y$. Moreover  ${\overline{\cL^*}} = \cL^\dagger=  \overline{\cL}^\dagger $ and thus $e^{\overline{\cL^*} t}=e^{\cL^\dagger t}  = (e^{\overline{\cL}t})^\dagger$ hold.
\end{proposition}
\begin{proof}
The fact that $\overline{\cL}$ generates a contraction semigroup under Assumption \eqref{ass:A1} is known from \cite{GroS16}. Since  adjoints are not addressed there we give some insight into the proof and in doing so verify the statements concerning $\cL^*$.
All proofs for the generation property of $\overline{\cL}$ make use of the operator $\mathcal{L}_{\mathrm{K}}$ defined on $C_0^\infty$ by
\begin{align}\label{eq:kfp}
 \mathcal{L}_{\mathrm{K}}\zeta =\Delta_v\zeta  - \tfrac{1}{4}\|v\|^2\zeta  + \tfrac{d}{2}\zeta -\nabla_x G^\top \nabla_v \zeta  + v^\top \nabla_x\zeta.
\end{align}
The operator ${\mathcal{L}_{\mathrm{K}}^\sharp}$ is obtained from ${\mathcal{L}_{\mathrm{K}}}$ by changing the signs of the last two summands of that operator.
A  computation shows that 
\begin{equation}\label{eq:aux9}
\cL =\mathcal{M}^{-1}\mathcal{L}_{\mathrm{K}}\mathcal{M}, \quad \text{and } \cL^* =\mathcal{M}^{-1}\mathcal{L}^\sharp_{\mathrm{K}}\mathcal{M}.
\end{equation}
Since $C_0^\infty$ is dense in $Y$ the operator $\mathcal{L}_{\mathrm{K}}$ is densely defined. Moreover by direct computation, it is dissipative in $Y$, and hence $\mathcal{L}_{\mathrm{K}}$ is closable, see \cite[Chapter II, 3.14 Proposition]{EngN99}.
Moreover, from \cite[Theorem 2.1]{ConG10} and its proof we know that  $\overline{\mathcal{L}_{\mathrm{K}}}$ is maximally dissipative in $L^2(\mathbb{R}^{2d})$.
By the first equality in \eqref{eq:aux9} one can argue that $\overline{\cL}$ is maximally dissipative in $Y$ and hence by the Lumer-Phillips theorem it generates a contraction semigroup $e^{\overline{\cL} t}$ on $Y$, see, e.g., \cite[Theorem 3.15]{EngN99}. As a side remark we mention that under the stronger regularity assumption $G\in C^\infty(\mathbb{R}^d)$ it was verified in \cite[Proposition 5.5]{HelN05} that $
\overline{\mathcal{L}_{\mathrm{K}}^\sharp}$ is maximally dissipative in $L^2(\mathbb{R}^{2d})$ and the proof there can readily be adopted to argue that $ \overline{\mathcal{L}_{\mathrm{K}}}$ is maximally dissipative as well.

Next we turn to adjoints. The proof of \cite[Theorem 2.1]{ConG10} can  be adapted to argue that  $\overline{\mathcal{L}_{\mathrm{K}}^\sharp}$ is maximally dissipative in $L^2(\mathbb{R}^{2d})$ under Assumption \eqref{ass:A1}.  Hence from the second equality in \eqref{eq:aux9} it follows that $\overline{\cL^*}$ is maximally dissipative in $Y$ and generates the semigroup $e^{\overline{\cL^*} t}$ on $Y$.

Let us further observe that the $Y$-adjoint $\overline{\cL}^\dagger = {\cL}^\dagger$ of $ \overline{\cL}$ is also maximally dissipative, \cite[Proposition 3.1.10]{TucW09}. The graph of $\overline{\cL^*}$ is contained in the graph of $\overline{\cL}^\dagger$. Since both operators are maximally dissipative we have $\overline{\cL^*}={\cL}^\dagger $ and $e^{\overline{\cL^*}t}=e^{{\cL}^\dagger  t} $. The final claim $(e^{\overline{\cL}t})^\dagger=e^{\cL^\dagger t}$ follows from a general result on adjoint semigroups, see \cite[Corollary 1.10.6]{Paz83}.
\end{proof}

\begin{remark}\label{rem:pH}{\em
 While not essentially needed later in this paper, let us make some observations on the relation between the operators  $\mathcal{L}^\sharp, \mathcal{L}^*$, and $\mathcal{L}^\sharp_{\mathrm{K}}$. For this purpose we introduce the following operators on $C^\infty_0(\mathbb{R}^{2d})$:
  \begin{align*}
  R_{\mu^-}&= -\mathrm{div}_v(\nabla_v \cdot + v\cdot),  \qquad R_{\mu}= -(\Delta_v\cdot - v^\top \nabla_v \cdot), \\ R_{L^2}&= -(\Delta_v \cdot - \tfrac{1}{4}\|v\|^2\cdot+ \tfrac{d}{2}\cdot), \qquad
   J= - v^\top \nabla_x  + \nabla_x G^\top \nabla_v.
   \end{align*}
   In this way we obtain
 \begin{equation*}
  \mathcal{L}=-J-R_{\mu}, \quad\mathcal{L}^\sharp=J-R_{\mu^-}, \quad \mathcal{L}^*=J-R_{\mu}, \quad \mathcal{L}^\sharp_{\mathrm{K}}=J-R_{L^2}.
   \end{equation*}
A short computation shows that $J$ is formally skew-adjoint in $L^2(\mathbb{R}^{2d})$ as well as in $Y$. Further we have that $\mathcal{M} J= J\mathcal{M}$ and  
\begin{equation*}
 \mathcal{L}^\sharp_{\mathrm{K}}=\mathcal{M}^{-1} \mathcal{L}^\sharp \mathcal{M}, \quad \mathcal{L}^*=\mathcal{M}^{-2} \mathcal{L}^\sharp \mathcal{M}^2.
  \end{equation*}
The $'R'$ summands  in   $\mathcal{L}^\sharp, \mathcal{L}^\sharp_{\mathrm{K}}$, and $\mathcal{L}^*$ are such that
$R_{\mu^-}$ is formally selfadjoint in $L^2_{\mu^{-1}}(\mathbb{R}^{2d})$,
$R_{L^2}$ is formally selfadjoint in $L^2(\mathbb{R}^{2d})$,
$R_{\mu}$ is formally selfadjoint in $L^2_{\mu}(\mathbb{R}^{2d})$,
and these operators are nonnegative.
}
\end{remark}

The following technical result will be needed in the next section.

\begin{lemma}\label{lem:domains}
  It holds that
  \begin{align*}
   \mathcal{D}(\overline{\mathcal{L}}) \subset H_{\mu,v}^1(\mathbb R^{2d}),  \quad
   \mathcal{D}(\overline{\mathcal{L}^*}) \subset H_{\mu,v}^1(\mathbb R^{2d}),
  \end{align*}
  with continuous injections and domains endowed with the graph norm. Morover,
\begin{align*}
 -\langle \overline{\mathcal{L}} \varphi,\varphi \rangle_{L^2_\mu} &= \| \nabla_v \varphi \|_{L^2_\mu}^2 \ \ \forall \varphi \in \mathcal{D}(\overline{\mathcal{L}}), \\
  -\langle \overline{\mathcal{L}^*} \phi,\phi \rangle_{L^2_\mu} &= \| \nabla_v \phi \|_{L^2_\mu}^2 \ \ \forall \phi \in \mathcal{D}(\overline{\mathcal{L}^*}).
\end{align*}
\end{lemma}

\begin{proof}
 By computation, see, e.g., \cite[pp.~719]{LelS16}, we know
 \begin{align}\label{eq:aux1}
  -\langle \mathcal{L} \varphi,\varphi \rangle_{L^2_\mu} = \| \nabla_v \varphi \|_{L^2_\mu}^2 \ \forall \varphi \in C_0^{\infty}(\mathbb R^{2d}).
 \end{align}
Let $\varphi \in \mathcal{D}(\overline{\mathcal{L}})$. Since $C_0^\infty(\mathbb R^{2d})$ is a core for $\overline{\mathcal{L}}$, there exists $\varphi_n \in C_0^{\infty}(\mathbb R^{2d})$ such that $\varphi_n \to \varphi$ and $\overline{\mathcal{L}}\varphi_n \to \overline{\mathcal{L}} \varphi$ in $L^2_\mu.$ From \eqref{eq:aux1} we deduce that $
% \begin{align*}
 \| \mathcal{L} \varphi_n \| _{L^2_\mu} \| \varphi_n \| _{L^2_\mu} \ge \| \nabla_v \varphi_n \| _{L^2_\mu}^2.$
% \end{align*}
Consequently, $\{\varphi_n\}_{n=1}^{\infty}$ is bounded in $H_{\mu,v}^1 (\mathbb R^{2d})$ and thus on a subsequence we have $\varphi_n \rightharpoonup  \tilde{\varphi}$ weakly in $H^1_{\mu,v}(\mathbb R^{2d})$ for some $\tilde{\varphi}\in H_{\mu,v}^1 (\mathbb R^{2n})$. Since $\varphi_n\to \varphi$ in $L^2_\mu(\mathbb R^{2d})$, we have $\varphi_n \rightharpoonup \varphi$ in $H_{\mu,v}^1(\mathbb R^{2d})$ and in particular $\varphi \in H_{\mu,v}^1(\mathbb R^{2d})$. Observe that \eqref{eq:aux1} implies that
\begin{align*}
 \| \mathcal{L}(\varphi_n-\varphi_m)\|_{L^2_\mu} \| \varphi_n -\varphi_m\|_{L^2_\mu} \ge \| \nabla  _v (\varphi_n -\varphi_m ) \| _{L^2_\mu}^2.
\end{align*}
Hence, $\{\varphi_n\}$ is a Cauchy sequence in $H_{\mu,v}^1(\mathbb R^{2d})$. It admits a strong limit in $H_{\mu,v}^1(\mathbb R^{2d})$ which is necessarily $\varphi$. Taking the limit in $
% \begin{align*}
 -\langle \mathcal{L}\varphi_n,\varphi_n \rangle_{L^2_\mu} = \| \nabla_v \varphi_n\| _{L^2_\mu}^2
% \end{align*}
$
we arrive at the first equality. The arguments for $\overline{\mathcal{L}^*}$ are analogous.
\end{proof}

\section{The controlled kinetic Fokker-Planck equation} \label{s3}

Returning to \eqref{eq:langevincontrol}, let us briefly comment on how the change in the potential affects the deterministic dynamics associated with $\mathcal{L}$ and $\mathcal{L}^*$. From \eqref{eq:l}, we conclude that for fixed time $t$ and control $u=u(t)$, the operator  $\mathcal{L}_{\mathrm{c}}(u)$ is given by
\begin{align}\label{eq:lu}
  \mathcal{L}_{\mathrm{c}}(u)\varphi =  v^\top \nabla_x \tcb\varphi- \nabla_x G^\top \nabla_v \varphi- u\nabla_x \alpha^\top \nabla_v \varphi -v^\top   \nabla _v \varphi + \Delta_v \varphi.
\end{align}
Throughout the remainder we make the following additional assumption.

\begin{customthm}{A2} \label{ass:A2}
 The control potential satisfies $\alpha \in W^{1,\infty}(\mathbb R^d)$.
\end{customthm}

For $\varphi,\phi \in C_0^\infty$, it holds that
\begin{equation}\label{eq:kk1}
\begin{aligned}
 \langle \phi, \mathcal{L}_c(u)\varphi \rangle_{L^2_\mu} &= \langle \phi, \mathcal{L}\varphi\rangle_{L^2_\mu} - u\int \mu \phi \nabla_x \alpha^\top \nabla_v \varphi \dxdv \\
 &= \langle \mathcal{L}^* \phi, \varphi \rangle_{L^2_\mu}+u \int \varphi \nabla_x \alpha^\top \nabla_v (\mu \phi)\dxdv \\
 &=\langle \mathcal{L}^* \phi, \varphi \rangle_{L^2_\mu}+u \!\int \mu \varphi \nabla_x \alpha^\top \nabla_v \phi \dxdv -u \!\int \mu \varphi \phi  \nabla_x \alpha ^\top v\dxdv, \\
\end{aligned}
\end{equation}
where we have used \eqref{eq:aux_1}. This leads to  the abstract bilinear controlled system
\begin{equation}\label{eq:abs_bil_cau}
\begin{aligned}
  \ddt y &= Ay + uNy , \quad y(0)=y_0
\end{aligned}
\end{equation}
where 
\begin{align*}
 Ay &= \overline{\mathcal{L}^*}y = -v^\top  \nabla_x y + (\nabla_x G)^\top \nabla_v y  -  v^\top   \nabla_v y+  \Delta_v y \\
 N y &=  (\nabla_x \alpha) ^\top \nabla _v y - y (\nabla_x \alpha)^\top  v  .
\end{align*}
Given $y,z \in  C_0^\infty(\mathbb R^{2d})$, it follows from the computations in \eqref{eq:kk1} that
\begin{equation*}
\begin{aligned}
 \langle z,Ny\rangle_Y &=\int \mu z (\nabla_x \alpha)^\top \nabla_v y \dxdv - \int \mu z y(\nabla _x \alpha)^\top v \dxdv \\&= \langle -(\nabla_x \alpha)^\top \nabla_v z,y\rangle _Y=: \langle N^*z,y\rangle_Y.
\end{aligned}
\end{equation*}
Note that $N$, being a first order differential operator, is not bounded on the state space $Y$. In the context of well-posed linear systems \cite{Sta05}, we consider it as a generator of a well-posed input map  for which we need the following estimate for each $y\in C_0^\infty(\mathbb R^{2d})$:
\begin{equation}\label{eq:N_H1}
\begin{aligned}
 \| Ny\| _{V_v'} &= \sup\limits_{z\in V_v, \| z\|_{V_v}=1 }   \langle Ny, z \rangle _Y =\sup\limits_{z\in V_v, \| z\|_{V_v}=1 } \langle y,N^*z \rangle_Y \\
 &=\sup\limits_{z\in V_v, \| z\|_{V_v}=1 } \| y\| _Y \| (\nabla_x \alpha)^\top \nabla_v z\|_Y \\
 &\le \sup\limits_{z\in V_v, \| z\|_{V_v}=1 }  C \| y\|_ Y \| z \| _{H_{v,\mu}^1(\mathbb R^{2d})}  \le C \| y\|_{Y},
\end{aligned}
\end{equation}
where we used that $\alpha\in W^{1,\infty}(\mathbb R^d)$.

Since $C_0^{\infty}(\mathbb R^{2d})$ is dense in $Y$, we obtain $N\in \mathcal{B}(Y,V_v')$ and $N\in \mathcal{B}(Y,V')$.
Moreover, $A^\dagger = \overline{\mathcal{L}}$ considered as an operator in $Y$, and therefore, with reference to  Lemma \ref{lem:domains}, we also find that  $N\in \mathcal{B}(Y,[\mathcal{D}(A^\dagger)]')$. This is the form in which $N$ will be interpreted in the following steps.

Before we turn to \eqref{eq:abs_bil_cau}, we first focus on the linear system
\begin{align}\label{eq:abs_lin_cau}
 \ddt{y}(t) = Ay(t) + Nw(t), \quad y(0)=y_0 \in Y,
\end{align}
where $w \in L_{\mathrm{loc}}^1([0,\infty);Y)$ is given. Following \cite[Definition 4.1.1]{TucW09}, a \emph{solution of \eqref{eq:abs_lin_cau} in $Y_{-1}:=\mathcal{D}(A^\dagger)'$} is a function
\begin{align*}
 y \in L^1_{\mathrm{loc}}(0,\infty;Y) \cap C([0,\infty);Y_{-1})
\end{align*}
which satisfies the following equation in $Y_{-1}$:
\begin{align}\label{eq:int_form}
y(t)- y_0  = \int_0^t A y(\sigma) + Nw(\sigma) \, \mathrm{d}\sigma \quad \forall t \in [0,\infty),
\end{align}
with $y_0\in Y$.
 In \eqref{eq:int_form} the operator $A$ represents the unique extension of $A\in\mathcal{B}(\mathcal{D}(A),Y)$ as operator in $\mathcal{B}(Y,Y_{-1})$. This extension can be identified with $ (A^{\dagger})'$, see \cite[Proposition 2.3.10]{TucW09}.    We  shall use no extra notation for this extension. Further it will be convenient to recall that \eqref{eq:int_form} is equivalent to
\begin{align}\label{eq:int_form2}
\langle y(t)-y_0, z \rangle_{\mathcal{D}  }  = \int_0^t [\langle y(\sigma) , A^\dagger z \rangle_{Y} + \langle Nw(\sigma), z\rangle_{\mathcal{D}  }]   \, \mathrm{d}\sigma,
\end{align}
for all $z\in \mathcal{D}(A^\dagger)$ and  $t \in [0,\infty)$, see \cite[Remark 4.1.2]{TucW09}.
It is well-known \cite[Proposition 4.1.4]{TucW09} that if $y$ is a solution of \eqref{eq:abs_lin_cau} in $Y_{-1}$, it is given by the unique mild solution
\begin{align}\label{eq:mild_solution_lin}
 y(t) = e^{At} y_0 + \int_0^t e^{A(t-\sigma)} Nw(\sigma) \, \mathrm{d}\sigma.
\end{align}
Next we aim for proving the existence of a solution $y$ and additional regularity.  For this purposes, let us recall the theory of admissible control and observation operators, see, e.g., \cite[Section 4]{TucW09}. The operator $N \in \mathcal{B}(Y,[\mathcal{D}(A^\dagger)]')$ is called an \emph{admissible control operator} for $e^{At}$ if for some $\tau > 0$, $\mathrm{Ran}\, \Phi_\tau \subset Y$, where $\Phi_\tau \in \mathcal{B}(L^2([0,\infty);Y),[\mathcal{D}(A^\dagger)]')$ is defined by
\begin{align}\label{eq:cont_map}
 \Phi_\tau w = \int_0^\tau e^{A(\tau-s)} Nw(s)\, \mathrm{d}s.
\end{align}
Similarly, given $C\in \mathcal{B}(\mathcal{D}(A),Z)$ with $Z$ a Hilbert space, we consider
\begin{align*}
 \ddt{y}(t)&= Ay(t), \quad y(0)=y_0,\quad y_{\mathrm{obs}}(t)=Cy(t)
\end{align*}
and call $C\in \mathcal{B}(\mathcal{D}(A),Z)$ an \emph{admissible observation operator} for $e^{At}$ if for some $\tau,$ $\Psi_\tau \in \mathcal{B}(\mathcal{D}(A),L^2([0,\infty);Z))$ has a continuous extension to $Y$. Here, $\Psi_\tau$ is defined by
\begin{equation}\label{eq:obsv_map}
\begin{aligned}
 (\Psi_\tau y_0)(t)= \begin{cases} Ce^{At}y_0 & \text{for } t \in [0,\tau], \\
                      0 & \text{for } t>\tau .
                     \end{cases}
\end{aligned}
\end{equation}
Equivalently (see \cite[Section 4.3]{TucW09}), $C\in \mathcal{B}(\mathcal{D}(A),Z)$ is an admissible observation operator for $e^{At}$ if and only if, for some $\tau >0$, there exists a constant $K_\tau \ge 0$ s.t.
\begin{align*}
 \int_0^\tau \| Ce^{At} y_0\|_Z^2 \,\mathrm{d}t \le K_\tau^2 \| y_0\|_Y^2 \quad \forall y_0 \in \mathcal{D}(A).
\end{align*}
From \cite[Theorem 4.4.3]{TucW09}, we know that $N\in \mathcal{B}(Y,[\mathcal{D}(A^\dagger)]')$ is an admissible control operator for $e^{At}$ if and only if $N' \in \mathcal{B}(\mathcal{D}(A^\dagger),Y)$ is an admissible observation operator for $(e^{At})^\dagger=e^{A^\dagger t}$.

\begin{proposition}\label{prop:mild_solution_lin} Let Assumptions \eqref{ass:A1} and \eqref{ass:A2} hold.
 Then  for every $y_0 \in Y$ and every $w \in L^2_{\mathrm{loc}}(0,\infty;Y)$, the initial value problem \eqref{eq:abs_lin_cau} has a unique solution in $Y_{-1}$. It  is given by \eqref{eq:mild_solution_lin} and it satisfies
  \begin{align*}
    y\in C([0,\infty);Y) \cap H^1_{\mathrm{loc}}(0,\infty;Y_{-1}).
  \end{align*}
\end{proposition}

\begin{proof}
By Proposition \ref{prop:adjoint_relation} and Lemma \ref{lem:domains}, the operator $\Pi=M\cdot \mathrm{id} \in \mathcal{B}(Y)$,  with $M>0$, satisfies $\Pi\ge 0 $ as well as
  \begin{align*}
  \langle \Pi z, A^\dagger z \rangle_Y &= M \langle z,(\overline{\mathcal{L^*}})^{\dagger}z \rangle_Y= M \langle z,\overline{\mathcal{L}} z \rangle_Y =-M\| \nabla_v z\|_Y^2  \quad \forall z \in \mathcal{D}(A^\dagger).
\end{align*}
On the other hand, for $N'$ and $z \in \mathcal{D}(A^\dagger)$, we obtain
\begin{align}\label{eq:test_N}
 \|N'z\|_Y^2 = \| \nabla_x \alpha^\top \nabla _v z\|_Y^2 \le C \| \nabla_v z\| _{Y}^2.
\end{align}
Hence, choosing $M>0$ large enough, it follows that
\begin{align*}
2   \langle \Pi z, A^\dagger z \rangle_Y \le -\| N'z\| _Y^2 \quad \forall z \in \mathcal{D}(A^\dagger),
\end{align*}
which, by \cite[Theorem 5.1.1]{TucW09}, implies that $N'$ is an infinite-time admissible observation operator for $e^{A^\dagger t}$. Hence, $N$ is an infinite-time admissible control operator for $e^{At}$ and the result then follows with \cite[Proposition 4.2.5]{TucW09}.
\end{proof}

We shall make  use of the following density result.

\begin{lemma}\label{le:density}
The space $\Cc$ is dense in $V$.
\end{lemma}

The proof is well known in the case without weights \cite[Theorem
3.18]{Ada75} and can be extended to the weighted case in a straightforward manner. Henceforth we will consider $N$ as an operator in $ \mathcal{B}(Y,V_v')$.

\begin{proposition}\label{prop:lin_more_reg}Let Assumptions \eqref{ass:A1} and \eqref{ass:A2} hold. Then
  for every $y_0 \in Y$ and every $w \in L^2_{\mathrm{loc}}(0,\infty;Y)$, the unique solution $y$ of \eqref{eq:abs_lin_cau} satisfies $\nabla_v y \in L^2_{\mathrm{loc}}(0,\infty;Y^d)$. Additionally, for all $T>0$ there exist constants $C_1,C_2(T),C_3(T)$ s.t.
  \begin{subequations}\label{eq:aprio_lin}
\begin{eqnarray}
  \mathrm{max}( \|y\|_{L^{\infty}(0,T;Y)}, \|\nabla_v y\|_{L^{2}(0,T;Y)})  &&\hspace{-0.5cm}\le C_1 \left( \| y_0\|_Y + \| w\|_{L^2(0,T;Y)} \right), \ \ \ \\
   \|y\|_{L^2(0,T;V_v)}  &&\hspace{-0.5cm}\le C_2(T) \left( \| y_0\|_Y + \| w\|_{L^2(0,T;Y)} \right), \qquad \ \ \ \\
   \|\dot y\|_{L^2(0,T;Y_{-1})}   &&\hspace{-0.5cm}\le C_3(T) \left( \| y_0\|_Y + \| w\|_{L^2(0,T;Y)} \right). \ \  \
\end{eqnarray}
\end{subequations}
\end{proposition}

\begin{proof}
We will show the assertion by introducing a suitable perturbation which provides the new operator $A_{\varepsilon}$ with a $Y$-$V$ coercivity property. As a consequence, we can resort to maximal regularity results of analytic semigroups to obtain (uniform) a priori estimates for the corresponding solution $y_\varepsilon$. By passing to the limit $\varepsilon \to 0$, we obtain similar estimates for the unperturbed solution $y$.

For $\varepsilon>0$, we define the bilinear form $a_{\varepsilon}\colon C^\infty_0(\mathbb{R}^{2d})\times C^\infty_0(\mathbb{R}^{2d}) \to \mathbb R$ by
\begin{equation}\label{eq:per_bil_form}
\begin{aligned}
 a_{\varepsilon}(y,z)&:= \int \nabla_v (\mu z)^\top \nabla_v y \dxdv+ \int \mu z (v^\top \nabla_x y-\nabla_x G^\top \nabla_v y + v^\top \nabla_v y) \dxdv  \\
 &\quad +\varepsilon \int \nabla_x(\mu z)^\top \nabla_x y \dxdv +\varepsilon \int \mu z \nabla_xG^\top \nabla _x y \dxdv.
\end{aligned}
\end{equation}
It arises from testing $Ay$ with $z$, adding artificial diffusion and integration by parts. Next we rearrange terms in  $a_{\varepsilon}$ and obtain
\begin{align*}
 a_\varepsilon(y,z)&= \int z\nabla _v \mu ^\top \nabla_v y \dxdv +\int \mu \nabla_v z^\top \nabla_v y\dxdv+\int \mu z v^\top \nabla_v y\dxdv \\
 &\qquad +\int \mu z(v^\top \nabla_x y - \nabla_x G^\top \nabla _vy)\dxdv  +\varepsilon \int z\nabla_x \mu^\top \nabla_x y \dxdv \\&\qquad  + \varepsilon \int \mu \nabla_x z^\top \nabla_x y \dxdv+ \varepsilon \int \mu z \nabla_x G^\top \nabla_x y \dxdv \\
 &= \int \mu \nabla_v z^\top \nabla_v y\dxdv + \varepsilon \int \mu \nabla_x z^\top \nabla_x y\dxdv \\
 &\qquad +\int \mu z(v^\top \nabla_x y - \nabla_x G^\top \nabla _vy)\dxdv.
\end{align*}
Considering the last term explicitly and using \eqref{eq:aux_1}, we find that
\begin{align*}
& \int \mu z (v^\top \nabla_x y - \nabla_x G^\top \nabla_v y) \dxdv = -\int z \nabla_v \mu^\top \nabla_x y \dxdv + \int z \nabla_x\mu^\top \nabla_v y \dxdv\\
 &\quad=\int \mu \mathrm{div}_v (z \nabla _x y) \dxdv - \int \mu \mathrm{div}_x (z \nabla_v y)\dxdv \\
 &\quad=\int \mu \nabla_v z^\top \nabla _x y\dxdv +\int\mu z \mathrm{div}_v (\nabla_x y)\dxdv \\&\qquad -\int \mu \nabla_x z^\top \nabla_v y\dxdv - \int \mu z \mathrm{div}_x(\nabla _v y)\dxdv \\
 &\quad= \int \mu \nabla_v z^\top \nabla _x y \dxdv -\int \mu \nabla_x z^\top \nabla_v y \dxdv.
\end{align*}
Altogether, for $y,z\in C_0^\infty(\mathbb{R}^{2d})$ we have established that
\begin{equation}\label{eq:ae_simplified}
\begin{aligned}
 a_\varepsilon(y,z)&=\int \mu \nabla_v z^\top \nabla_v y \dxdv + \varepsilon \int \mu \nabla_x y^\top \nabla_x z \dxdv \\&\qquad  + \int \mu \nabla_v z^\top \nabla _x y \dxdv -\int \mu \nabla_x z^\top \nabla_v y \dxdv.
\end{aligned}
\end{equation}
Utilizing density of $C_0^\infty(\mathbb{R}^{2d})$ in $V$ it follows that
$a_\varepsilon$ can be extended to a bounded bilinear  form on $V \times V$. With $a_{\varepsilon}$ we associate the operator $A_{\varepsilon}$ in $Y$ by
\begin{align*}
 \mathcal{D}(A_{\varepsilon})&=\left\{ y \in V\ | \ z\mapsto a_{\varepsilon}(y,z) \text{ is $Y$-continuous} \right\}, \\
 \langle A_{\varepsilon} y,z\rangle_Y &= - a_{\varepsilon}(y,z), \ \ \forall \,y \in\mathcal{D}(A_\varepsilon),  \ z \in V.
\end{align*}
Without change of notation we also consider the extension $A_{\varepsilon}\in \mathcal{B}(V,V')$ defined by
$
\langle A_{\varepsilon} y,z\rangle_{V',V} = - a_{\varepsilon}(y,z)$ for all $y,z \in V.$
As an aside we  note that for $y,z \in \Cc$ it holds that
\begin{align*}
 \langle A_\varepsilon y,z \rangle _Y &= \langle Ay,z\rangle_Y + \langle R_\varepsilon y,z\rangle_Y,
\end{align*}
where the action of $R_\varepsilon$ is defined by
\begin{align}\label{eq:action_R}
 R_{\varepsilon} y =\varepsilon \Delta_x y- \varepsilon \nabla_x G^\top \nabla_x y.
\end{align}
For $y \in \Cc$ we obtain, utilizing that $G\in W_{\mathrm{loc}}^{1,\infty}(\mathbb{R}^d)$,
\begin{equation}\label{eq:test_Ae}
\begin{aligned}
 -a_{\varepsilon}(y,y)&=\langle A_{\varepsilon} y,y\rangle_Y = \langle \overline{\mathcal{L}^*} y,y\rangle_Y + \varepsilon \int \mu y\Delta_x y \dxdv- \varepsilon \int \mu y\nabla_x G^\top \nabla_x y \dxdv \\
 &=-\| \nabla_v y\|_Y^2 - \varepsilon \int \nabla_x( \mu y)^\top \nabla_x y \dxdv- \varepsilon \int \mu y\nabla_x G^\top \nabla_x y \dxdv\\
 &=-\| \nabla_v y\|_Y^2 - \varepsilon \int y\nabla_x \mu^\top \nabla_x y \dxdv \\&\quad- \varepsilon \int \mu\nabla_x y^\top \nabla_x y \dxdv- \varepsilon \int \mu y\nabla_x G^\top \nabla_x y \dxdv\\
 &= -\|\nabla_v y \| _Y ^2 - \varepsilon\| \nabla_x y \| _Y^2,
\end{aligned}
\end{equation}
where in  the last step we referred to  \eqref{eq:aux_1}. Using density of $\Cc$ in $V$, we conclude that \eqref{eq:test_Ae} holds for all $y\in V$ and thus  \emph{$V$-$Y$ coercivity} of $a_\varepsilon$ follows. Hence $A_\varepsilon$ generates an analytic semigroup in $Y$, see \cite[Part II, Chapter 1, Theorem 2.12]{Benetal07}. We next  consider the perturbed variational equation in $V'$
 \begin{align}\label{eq:per_lin}
  \ddt{y}_{\varepsilon}(t)=A_{\varepsilon}y_{\varepsilon}(t)+Nw(t), \ \ y_{\varepsilon}(0)=y_0 \in Y,
 \end{align}
for which we conclude from (\cite[Part II, Chapter 2, Theorem 1.1]{Benetal07}) that it admits a unique solution
\begin{align}\label{eq:per_sol}
 y_\varepsilon \in L^2(0,T;V) \cap H^1(0,T;V'). %\cap L^\infty(0,T;Y)
\end{align}
Taking the inner product with $\mu y_\varepsilon$ and integrating over $\mathbb R^{2d}$ yields
\begin{align*}
 \langle  \tfrac{\mathrm{d}}{\mathrm{d}t} y_\varepsilon(t),y_\varepsilon(t) \rangle _{V',V}  = \langle A_\varepsilon y_\varepsilon(t),y_\varepsilon(t) \rangle _{V',V}  + \langle Nw(t),y_\varepsilon(t)\rangle_{V_v',V_v},
\end{align*}
for almost every $t \in (0,T)$. With \eqref{eq:test_N} and \eqref{eq:test_Ae}, and the computation in \eqref{eq:abs_bil_cau} it  follows that
\begin{equation}\label{eq:aprio_aux}
\begin{aligned}
 \tfrac{1}{2}\tfrac{\mathrm{d}}{\mathrm{d}t} \| y_\varepsilon \|_Y^2
&=\langle A_\varepsilon y_\varepsilon,y_\varepsilon \rangle_{V',V}  + \langle w,N'y_\varepsilon\rangle_Y  = -\|\nabla_v y_\varepsilon\|_Y^2 - \varepsilon \|\nabla_x y_\varepsilon \|_Y^2 +\| w\|_Y   \| N'y_\varepsilon\|_Y \\
 &\le -\|\nabla_v y_\varepsilon\|_Y^2 - \varepsilon \|\nabla_x y_\varepsilon \|_Y^2 + \| \alpha\|_{W^{1,\infty}(\mathbb R^{d})} \| w\|_Y \| \nabla _v y_\varepsilon\|_Y   \\
 &\le -\tfrac{1}{2}\|\nabla_v y_\varepsilon\|_Y^2 - \varepsilon \|\nabla_x y_\varepsilon \|_Y^2 + \tfrac{1}{2}\| \alpha\|^2_{W^{1,\infty}(\mathbb R^{d})}  \| w\|_Y^2,
\end{aligned}
\end{equation}
where we suppress the dependence on $t$. This  implies that
\begin{align*}
&\| y_\varepsilon(t)\|_Y^2 +\| \nabla_v y_\varepsilon\|_{L^2(0,t;Y)}^2 + 2 \varepsilon \| \nabla_x y_\varepsilon\|_{L^2(0,t;Y)}^2 \\
&\qquad \le \| y_\varepsilon(0)\|_Y^2  +\| \alpha\|^2_{W^{1,\infty}(\mathbb R^{d})} \int_0^t \|w(s)\|_Y^2 \, \mathrm{d}s
\end{align*}
and, hence, we obtain for a constant $M$, independent of $T$, that
\begin{equation}\label{eq:aprio_aux3}
\begin{aligned}
\max \left( \| y_\varepsilon \| _{L^\infty(0,T;Y)}, \| \nabla_v y_\varepsilon \| _{L^2(0,T;Y)}\right)&\le M \left( \|y_0\|_Y + \| w\|_{L^2(0,T;Y)} \right), \\[1ex]
\qquad \sqrt{\varepsilon} \| \nabla_x y_\varepsilon\|_{L^2(0,T;Y)} &\le M\left( \|y_0\|_Y + \| w\|_{L^2(0,T;Y)} \right).
\end{aligned}
\end{equation}

Next we pass to the limit as $\varepsilon \to 0$. Due to \eqref{eq:per_sol} we have that $y_\varepsilon\in C([0,T],Y)$, so that pointwise evaluation at $t\in [0,T]$ is justified. Moreover, the extension of $A$ to  $Y$ satisfies $A= (A^{\dagger})' \in \mathcal{B}(Y,Y_{-1})$, and thus  $A y_{\varepsilon}\in C([0,T];Y_{-1})$. Further $y_{\varepsilon}$ satisfies
\begin{equation}\label{eq:aux3a}
\langle \tfrac{\mathrm{d}}{\mathrm{d}t} y_\varepsilon(t),\varphi \rangle_{V',V} = \langle  A_\varepsilon  y_\varepsilon(t) + N w(t) ,\varphi \rangle_{V',V}
\end{equation}
for a.e. $t\in(0,T)$ and every $\varphi \in V$. For $\varphi\in V\cap \mathcal{D}(A^{\dagger})$ we have
$$\langle  A y_\varepsilon(t),\varphi \rangle_{V',V} =\langle  A y_\varepsilon(t),\varphi \rangle_{\mathcal{D}(A^{\dagger})',\mathcal{D}(A^{\dagger})} =  \langle  y_\varepsilon(t), A^\dagger \varphi \rangle_{Y}$$ and hence together with $y_\varepsilon\in C([0,T],Y)$ and \eqref{eq:ae_simplified}, we find
\begin{equation*}\label{eq:aux3}
\begin{array}{ll}
\langle y_\varepsilon(t) -y_0,\varphi(t)\rangle_Y\!\! &= \int_0^t \langle y_\varepsilon(s), A^{\dagger}\varphi(t) \rangle_Y \,\mathrm{d}s + \int_0^t \langle N w(s), \varphi(t) \rangle_{V',V} \,\mathrm{d}s  \\[1.5ex]
&- \varepsilon \int_0^t \langle \nabla_x y_\varepsilon(s),\nabla_x \varphi(t)\rangle_Y \,\mathrm{d}s,
\end{array}
\end{equation*}
for every $t\in[0,T]$. Consequently we obtain for every  $\varphi\in L^2(0,T; V\cap \mathcal{D}(A^{\dagger}))$:
\begin{equation}\label{eq:aux4}
\begin{array}{ll}
\int_0^T \langle y_\varepsilon(t)- y_0,\varphi(t)\rangle_Y \,\mathrm{d}t  = &\int _0^T\int_s^T \big( \langle y_\varepsilon(s), A^{\dagger}\varphi(t) \rangle_Y + \langle   N w(s), \varphi(t)   \rangle_{V',V}   \big) \, \mathrm{d}t\, \mathrm{d}s \\[1.5ex]
&- \varepsilon \int _0^T\int_s^T \langle \nabla_x y_\varepsilon(s),\nabla_x \varphi(t)\rangle_Y \,\mathrm{d}t\,\mathrm{d}s.
\end{array}
\end{equation}
Using \eqref{eq:aprio_aux3} we obtain:
\begin{equation*}
\begin{aligned}
&\varepsilon | \int _0^T\int_s^T \langle \nabla_x y_\varepsilon(s),\nabla_x \varphi(t) \rangle_Y \, \mathrm{d}t\, \mathrm{d}s | \\
&\quad \le \varepsilon T (\int^T_0|\nabla_x y_{\varepsilon}(s)|^2_Y   \, \mathrm{d}s )^{\frac{1}{2}} (\int^T_0|\nabla_x \varphi(s)|^2_Y   \, \mathrm{d}s )^{\frac{1}{2}} \to 0
\end{aligned}
\end{equation*}
for $\varepsilon \to 0$.

By \eqref{eq:aprio_aux3} there exists a null-sequence $\{\varepsilon_k\}_{k=1}^\infty$, and $y\in L^2(0,T;V_v)$ such that
\begin{equation*}
\lim\limits_{k\to \infty} y_{\varepsilon_k} \rightharpoonup y \text{ in } L^2(0,T;V_v).
\end{equation*}
Utilizing that  $\int _0^T\int_s^T  \langle y_\varepsilon(s), A^{\dagger}\varphi(t) \rangle_Y \mathrm{d}t\, \mathrm{d}s = \int _0^T\ \langle y_\varepsilon(s),\int_s^T A^{\dagger}\varphi(t)\,\mathrm{d}t\,  \rangle_Y \mathrm{d}s $
we can pass to the limit in \eqref{eq:aux4} with $\varepsilon= \varepsilon_k$ and obtain that
\begin{equation}\label{eq:aux10}
\begin{array}{l}
\int_0^T \big\langle \, y(t) - y_0 - \int _0^t A y(s) ,\varphi(t) \, \big \rangle_{\mathcal{D}(A^\dagger)',\mathcal{D}(A^\dagger)} \, \mathrm{d}s\, \mathrm{d}t\\[1.5ex]
\qquad - \int_0^T \int _0^t \langle Nw(s), \varphi(t) \rangle_{V',V} \, \mathrm{d}s\, \mathrm{d}t  = 0,
\end{array}
\end{equation}
for all $w\in L^2(0,T; V\cap \,\mathcal{D}(A^{\dagger}))$. Since $A^\dagger=\overline{\cL}$ by Proposition \ref{prop:adjoint_relation} it follows that $C^\infty_0(\mathbb{R}^d)$ is a core for $A^\dagger$ and thus $C^\infty_0(\mathbb{R}^d)$ is dense in the graph norm of $A^\dagger$ in $\mathcal{D}(A^{\dagger})$.
Hence $L^2(0,T; C^{\infty}_0(\Omega))$ is dense in $ L^2(0,T; \mathcal{D}(A^{\dagger}))$. Together with the fact  that $N\in  \mathcal{B}(Y,\mathcal{D}(A^\dagger)')$ equation \eqref{eq:aux10} implies that
\begin{equation}\label{eq:aux5}
\int_0^T \big\langle \, y(t) - y_0 - \int _0^t (A y(s) + N w(s)) \, \mathrm{d}s ,\varphi(t) \, \big \rangle_{\mathcal{D}(A^\dagger)',\mathcal{D}(A^\dagger)} \,  \mathrm{d}t = 0,
\end{equation}
 and consequently we have
\begin{equation}\label{eq:aux6}
 y(t)  - y_0 - \int _0^t A y(s) + N w(s) \,\mathrm{d}s = 0, \text{ in } Y_{-1}   \text{ and a.e. } t \in (0,T).
\end{equation}
Since $y\in L^2(0,T;Y)$ we have that $Ay\in L^2(0,T;Y_{-1})$ and thus $y$ is absolutely continuous with values in $Y_{-1}$. By the discussion below \eqref{eq:abs_lin_cau} therefore, the  limit $y$ of $y_{\varepsilon_k}$ is the unique solution to \eqref{eq:int_form}. Consequently the whole family $y_{\varepsilon}$ has $y$ as its limit.  The a priori estimates \eqref{eq:aprio_lin} follow from \eqref{eq:aprio_aux3} and \eqref{eq:aux6}.
\end{proof}

We now turn to  well-posedness of equation  \eqref{eq:abs_bil_cau} by combining the results of \cref{prop:lin_more_reg} with classical fixed point type arguments that can be found similarly in \cite{BalMS04} or \cite{HosJS18}. The proofs are, by now, rather standard and are therefore deferred to the appendix.

For $t_0$ sufficiently small and $u\in L^2(0,t_0)$, the next theorem justifies the introduction of a \emph{mild solution} of \eqref{eq:abs_bil_cau} which for $0\le t\le t_0$ satisfies the relation
\begin{align}\label{eq:mild_solution_bilinear}
 y(t) = e^{At}y_0 + \int_0^t e^{A(t-s)}Ny(s)u(s)\, \mathrm{d}s.
\end{align}

\begin{theorem}[Local mild solution]\label{thm:local_mild_solution} Let Assumptions \eqref{ass:A1} and \eqref{ass:A2} hold. Then for each $y_0 \in Y$ and $u \in L^2(0,T)$ there exists $t_0 \in (0,T]$, such that \eqref{eq:abs_bil_cau} has a unique mild solution $y \in C([0,t_0];Y)$.
\end{theorem}

In the following, we use standard Gronwall type estimates for $y$ to show that the maximum interval of existence is unbounded.

\begin{theorem}[Global mild solution]\label{thm:global_mild_solution} Let Assumptions \eqref{ass:A1} and \eqref{ass:A2} hold. Then  \eqref{eq:abs_bil_cau} has a unique mild solution in $C([0,T];Y)$ for every $T>0$.
\end{theorem}

From \eqref{eq:mild_solution_bilinear} and Theorem \ref{thm:global_mild_solution}, the function  $y$ is also the unique mild solution to the equation
\begin{align*}
 \ddt y(t)&= Ay(t) + N w(t), \ \ y(0)=y_0,
\end{align*}
where $w(\cdot):=y(\cdot)u(\cdot) \in L^2(0,T;Y)$. Thus $y$ satisfies
\begin{align*}
 y(t) = e^{At} y_0 + \int_0^t e^{A(t-\sigma)} u(\sigma) Ny(\sigma) \, \mathrm{d}\sigma,
\end{align*}
with $u\in L^2(0,T)$, and $A=\overline{\mathcal{L}^*}$. Thus Proposition \ref{prop:lin_more_reg} implies that $y$ also satisfies
\begin{align}\label{eq:int_form3}
\langle y(t)-y_0, z \rangle_{\mathcal{D}(A^\dagger)', \mathcal{D}(A^\dagger)  }  = \int_0^t [\langle y(\sigma) , A^\dagger z \rangle_{Y} + u(\sigma)\langle Ny(\sigma), z\rangle_{\mathcal{D}(A^\dagger)', \mathcal{D}(A^\dagger)  }]   \, \mathrm{d}\sigma,
\end{align}
for all $z\in \mathcal{D}(A^\dagger)$, where $A=(A^\dagger)' =\overline{\mathcal{L}}'$.

\section{Decoupling of the invariant measure}\label{s4}

Obviously, due to the dissipativity of $\mathcal{L}^*$, the eigenvalues of the operator ${A}=\overline{\mathcal{L}^*}$ are in the closed left half complex plane. Moreover, $0$ is an eigenvalue of ${A}$ with its associated eigenspace containing constant functions. To show that the constant functions are the only functions in the kernel and that convergence to equilibrium is exponential, we need the following conditions for the potential $G$.

\begin{customthm}{A3}\label{ass:A3}
 The regularity $G\in C^2(\mathbb{R}^d)$ holds.   Moreover the Hessian satisfies   $|\nabla^2 G(x)| \le \rho(1+|\nabla G(x)|)$ for some $\rho$ independent of $x \in  \mathbb{R}^d$ or $\nabla^2 G\in L^\infty(\mathbb R^{d\times d})$,   and a Poincar\'{e} inequality holds for the measure $$\nu (\mathrm{d}x) = Z_\nu^{-1} e^{-G(x)} \,\mathrm{d}x, \quad Z_\nu = \int_{\mathbb R^d} e^{-G(x)} \,\mathrm{d}x.  $$
\end{customthm}
Here, the measure $\nu$ is said to satisfy a Poincar\'{e} inequality with a constant $r>0$ if
\begin{align*}
 \| \varphi \| _{L^2(\nu)}^2 \le \tfrac{1}{2r} \| \nabla \varphi \| ^2 _{L^2(\nu)}, \quad \text{for all } \varphi \in H^1(\nu) \cap L^2_0(\nu),
\end{align*}
where
\begin{align*}
  L_0^2(\nu) &= \left \{ \varphi \in L^2(\nu) \ | \ \int _{\mathbb R^d} \varphi\, \mathrm{d}\nu = 0 \right\} , \
  H^1(\nu) = \left \{ \varphi \in L^2(\nu) \ | \ \nabla \varphi \in (L^2(\nu))^d \right\}.
\end{align*}
For a discussion of sufficient conditions for the validity of the Poincar\'{e} inequality see, e.g.,  \cite[Section 2.2]{LelS16}.

\begin{proposition}\label{prop:conv_to_equi}
Under Assumptions \eqref{ass:A1} and \eqref{ass:A3}, there exist constants $C$ and $\bar{\kappa}>0$ such that for all $\varphi \in L_0^2(\mu)$
\begin{align}\label{eq:exp_decay}
 \| e^{t A} \varphi \|_Y \le Ce^{-\bar{\kappa}t} \| \varphi\|_Y \quad \text{for all } t\ge 0.
\end{align}
\end{proposition}

\begin{proof}
 With Assumption \eqref{ass:A3} holding, it follows  from \cite[Proposition 2.20]{LelS16} or \cite{GroS16}, that  for all $\varphi \in L_0^2(\mu)$
 \begin{align*}
  \| e^{t A^\dagger} \varphi \|_Y \le Ce^{-\bar{\kappa}t} \| \varphi\|_Y \quad \text{for all } t\ge 0
 \end{align*}
 where we have used the relation $\bar{\mathcal{L}}=(\overline{\mathcal{L}^*})^\dagger$, see Proposition \ref{prop:adjoint_relation}. Since $(e^{At})^\dagger=e^{A^\dagger t}$ this  implies that
 \begin{align*}
  \| e^{t A} \varphi \|_Y = \| (e^{t A})^\dagger  \varphi \|_Y = \| e^{t A^\dagger} \varphi \| _Y \le Ce^{-\bar{\kappa}t} \| \varphi\|_Y  \quad \text{for all } t\ge 0.
 \end{align*}
\end{proof}

\begin{remark}
The dependence of $\bar \kappa$ and $C$ in \eqref{eq:exp_decay} on $\beta$ and $\gamma$ in \eqref{eq:langevin}, has been addressed in many publications including \cite{ArnE14,GroS16,LelS16,Vil09}. The estimate in \cite{GroS16} shows that $\bar \kappa\to 0^+$ as $\gamma\to 0^+$ or $C\to 1^+$.
\end{remark}

As an immediate consequence of this proposition, the eigenspace associated to the eigenvalue $0$ is one-dimensional. It is spanned by the constant function  $\mathbbm{1}$.  The associated eigenspace for the operator $\cL^\sharp$ is spanned by $\mu  \mathbbm{1}$.   We shall show that the control does not affect the dynamics of \eqref{eq:abs_bil_cau} on $span \mathbbm{1}$.  This suggests a splitting of the state space
according to the measure and its complement for which we introduce the (orthogonal) projection operators
\begin{equation}\label{eq:proj_op}
 \begin{aligned}
  P&\colon Y\to Y_0 = L_0^2(\mu), \ \ y\mapsto Py=y-\int_{\mathbb R^{2d}} y\mu \,\mathrm{d}x\,\mathrm{d}v \ \mathbbm{1},\\
  \mathcal{P}&\colon \mathcal{D}(A^\dagger)\to \mathcal{D}(A^\dagger)\cap Y_0, \ \ y\mapsto \mathcal{P}y=P_{|\mathcal{D}(A^\dagger)}y,\\
  Q&\colon Y\to Y_0^{\perp}, \ \ y\mapsto Qy=(I-P)y=\int_{\mathbb R^{2d}} y\mu \,\mathrm{d}x\,\mathrm{d}v \ \mathbbm{1},\\
  \mathcal{Q}&\colon \mathcal{D}(A^\dagger)\to Y_0^{\perp},\ \ y\mapsto \mathcal{Q}y=Q_{|\mathcal{D}(A^\dagger)}y.
 \end{aligned}
\end{equation}
With these projections, we may replace \eqref{eq:int_form3} as stated in the following proposition.
\begin{proposition}\label{eq:decoupled_sol}Let Assumptions \eqref{ass:A1}-\eqref{ass:A3} hold.
Then the function $y$ is a solution to \eqref{eq:abs_bil_cau} in the sense of equation \eqref{eq:int_form3} if and only if $y=Py+Qy=y_1+y_2$, where for all $z\in \mathcal{D}(A^\dagger)\cap Y_0$:
  \begin{subequations}\label{eq:split_total}
\begin{alignat}{4}
 \langle y_1(t)-Py_0,z \rangle_Y &= \int_0^t \langle y_1(\sigma) ,A^\dagger z \rangle _Y + u(\sigma) \langle y_1(\sigma)+Qy_0,N'z \rangle_Y \,\mathrm{d}\sigma,\label{eq:Psub} \\
  y_2(t)&=Qy_0 \quad  \forall t\ge 0.\label{eq:Qsub}
  \intertext{In particular, we have that}
   \int_{\mathbb R^{2d}} y(t) \, \mathrm{d}\mu &= \int_{\mathbb R^{2d}} y_0\, \mathrm{d}\mu\quad \forall t\ge 0.\label{eq:Qsub2}
\end{alignat}
\end{subequations}
\end{proposition}
\begin{proof}
 Let us assume that \eqref{eq:int_form3} holds for all $z \in \mathcal{D}(A^\dagger)$.  With regard to the integral representation
\eqref{eq:int_form3} of the solution, we split $y(t)\in Y$ and $z\in \mathcal{D}(A^\dagger)$ according to $y(t)=Py(t)+Qy(t)$ and $z=\mathcal{P}z+\mathcal{Q}z$. We then obtain
\begin{equation}\label{eq:aux7}
\begin{aligned}
 &\langle P(y(t)-y_0)+Q(y(t)-y_0),\mathcal{P}z+\mathcal{Q}z\rangle_{\mathcal{D}} =\langle y(t)-y_0,z\rangle_{\mathcal{D}}\\
 &=\int_0^t \langle y(\sigma),A^\dagger z \rangle_Y + u(\sigma) \langle Ny(\sigma),z\rangle _{\mathcal{D}} \, \mathrm{d}\sigma =\int_0^t \langle y(\sigma),A^\dagger z \rangle_Y + u(\sigma) \langle y(\sigma),N'z\rangle _Y\, \mathrm{d}\sigma \\
 &=\int_0^t \langle Py(\sigma)\!+\!Qy(\sigma),A^\dagger (\mathcal{P}z\!+\!\mathcal{Q}z) \rangle_Y \! +\! u(\sigma) \langle Py(\sigma)\!+\!Qy(\sigma),N'(\mathcal{P}z\!+\!\mathcal{Q}z)\rangle _Y\, \mathrm{d}\sigma.
\end{aligned}
\end{equation}
Since $\mathcal{Q}z \in Y_0^\perp$ is a constant function in $\mathbb R^{2d}$ we obtain
\begin{align}\label{eq:aux8}
 \langle \tilde{z}, A^\dagger \mathcal{Q}z\rangle_Y = 0 = \langle \tilde{z}, N' \mathcal{Q}z\rangle_Y \quad \forall \tilde{z} \in Y.
\end{align}
Moreover, as $P$ and $Q$, respectively $\mathcal{P}$ and $\mathcal{Q}$ are complementary projections, for all $\tilde{z}\in Y$ and $z\in\mathcal{D}(A^\dagger)$, it  holds that
\begin{align*}
& \langle P\tilde{z}+Q\tilde{z}, \mathcal{P}z + \mathcal{Q}z \rangle_{\mathcal{D}}=
 \langle P\tilde{z}+Q\tilde{z}, \mathcal{P}z + \mathcal{Q}z \rangle_Y \\
 &=\langle \tilde{z},P \mathcal{P}z \rangle_Y + \langle \tilde{z},P \mathcal{Q}z \rangle_Y + \langle \tilde{z},Q \mathcal{P}z \rangle_Y +\langle \tilde{z},Q \mathcal{Q}z \rangle_Y = \langle P\tilde{z},\mathcal{P}z \rangle_Y \!+\! \langle Q\tilde{z} , \mathcal{Q}z \rangle _Y.
\end{align*}
Returning to \eqref{eq:aux7}, we therefore conclude that
\begin{align*}
 &\langle P(y(t)-y_0),\mathcal{P}z \rangle_Y+ \langle Q(y(t)-y_0),\mathcal{Q}z \rangle_Y \\
 &= \int_0^t \langle Py(\sigma)+Qy(\sigma), A^\dagger \mathcal{P}z \rangle_Y + u(\sigma) \langle Py(\sigma)+Qy(\sigma),N'\mathcal{P}z \rangle _Y\, \mathrm{d}\sigma.
\end{align*}
Let us now consider two specific cases for $z\in \mathcal{D}(A^\dagger)$. \\[1ex]
\emph{Case 1:} $z\in Y_0^\perp$. Here, since $\mathcal{P}z=0$ and $\mathcal{Q}z=z$, it follows that
\begin{align*}
 \langle y(t)-y_0,z \rangle_Y = 0 \quad \forall z \in Y_0^\perp.
\end{align*}
This particularly yields
\begin{align*}
 Q(y(t)-y_0)=\int_{\mathbb R^{2d}} (y(t)-y_0)\mu \, \mathrm{d}x\,\mathrm{d}v \ \mathbbm{1} = \langle y(t)-y_0,\mathbbm{1} \rangle_Y \, \mathbbm{1}=0
\end{align*}
or, equivalently, $Qy(t) = Qy_0$ for all $t\ge 0$ which yields \eqref{eq:Qsub}. \\[1ex]
\emph{Case 2:} $z\in \mathcal{D}(A^\dagger)\cap Y_0$. Here, with $\mathcal{P}z=z$ and $\mathcal{Q}z=0$, it follows that
\begin{align*}
 \langle y(t)-y_0,z \rangle_Y = \int_0^t \langle Py(\sigma)\! + \!Qy(\sigma),A^\dagger z \rangle _Y + u(\sigma) \langle Py(\sigma)\!+\!Qy(\sigma),N'z \rangle_Y \,\mathrm{d}\sigma.
\end{align*}
Note that $Qy(\sigma)\in Y_0^\perp \subset \mathcal{D}(A)$  and $AQy(\sigma)=0$ for all $\sigma \in [0,t]$ which yields
\begin{align*}
 \langle  P y(t)- P y_0,z \rangle_Y = \int_0^t \langle Py(\sigma) ,A^\dagger z \rangle _Y + u(\sigma) \langle Py(\sigma)+Qy(\sigma),N'z \rangle_Y \,\mathrm{d}\sigma,
\end{align*}
and thus \eqref{eq:Psub} holds.  Note that we may also interpret this as an abstract equation in $[\mathcal{D}(A^\dagger)\cap Y_0]'$: 
\begin{align*}
 Py(t) = Py_0 + \int_0^t APy(\sigma) +u(\sigma)NPy(\sigma) + u(\sigma)NQy(\sigma) \, \mathrm{d}\sigma.
\end{align*}
 It can be verified by analogous arguments that \eqref{eq:Psub} and \eqref{eq:Qsub} imply \eqref{eq:int_form3}.
\end{proof}

\begin{remark}\label{rem:op_eq}
Note that the operator formulation of \eqref{eq:Psub} is given as an equation in $[\mathcal{D}(A^\dagger)\cap Y_0 ]'$ by
\begin{align}\label{eq:op_eq}
  \ddt{\widehat{y}} &= \widehat{A}\widehat{y} + u\widehat{N}\widehat{y}+ uNQy_0 , \quad \widehat{y}(0)=Py_0,
\end{align}
where $\widehat{A}= A_{|Y_0}, \widehat{N}= N_{|Y_0}$.
\end{remark}

Henceforth an additional assumption on the potential $G$ will be required \cite{HerN04}, where  we use the notation $\langle x \rangle = \sqrt{1+\|x\|^2}$.

\begin{customthm}{A4}\label{ass:A4}
(a)  The potential $G$ is a $C^\infty(\mathbb{R}^d)$ function, and there exists  $n> \frac{1}{2}$, and for each multi-index $\mathbbm{j}\in \mathbb{R}^d$   a positive constant $C_{\mathbbm{j}}$,  such that
\begin{equation*}
|\partial^\mathbbm{j}_x G(x)| \le C_\mathbbm{j}(1+\langle x\rangle^{2n-\min\{\|\mathbbm{j}\|,2\}}), \;\forall x\in \mathbb{R}^d.
\end{equation*}
(b) There exist constants $C_0>0$ and $C_1>0$ such that
\begin{equation*}
\pm G(x) \ge C_0^{-1} \langle x\rangle^{2n} - C_0  \text{ and } |\partial_x G(x)|\ge C_1^{-1} \langle x\rangle^{2n-1} - C_1.
\end{equation*}
\end{customthm}

\begin{proposition}\label{prop:sg_compact}
Let Assumptions \eqref{ass:A3}-\eqref{ass:A4} hold.  Then the semigroup $e^{\widehat{A}t}$ generated by $\widehat{A}=A_{|Y_0}$ is compact.
\end{proposition}
\begin{proof}
From the proof of Proposition  \ref{prop:adjoint_relation} it is known that $\overline{\mathcal{L}^\sharp_{\mathrm{K}}}$ is maximally dissipative and that it generates a semigroup $e^{\overline{\mathcal{L}^\sharp_{\mathrm{K}}}t}$.
As has been shown in \cite[Theorem 4.2]{HerN04}, this semigroup has a smoothing effect which, by the compact embeddings mentioned in the proof of Theorem 3.1 of the same article, is compact in $L^2(\mathbb{R}^{2d})$. Since $A=\mathcal{M}^{-1}\overline{\mathcal{L}^\sharp_{\mathrm{K}}}\mathcal{M}$, the semigroups $e^{At}$ and $e^{\overline{\mathcal{L}^\sharp_{\mathrm{K}}}t}$ are isomorphic with $e^{At}=\mathcal{M}^{-1} e^{\overline{\mathcal{L}^\sharp_{\mathrm{K}}}t} \mathcal{M}$, see, e.g., \cite[Section 5.10]{EngN99}. With \cite[Chapter 3, Theorem 4.8]{Kat80}, it follows that $e^{At}$ is compact in $Y$. Hence, it remains to be shown that $e^{\widehat{A}t}=(e^{At})_{|Y_0}$ is compact. This however follows by utilizing the orthogonal decomposition of the state (Hilbert) space $Y=Y_0\oplus Y_0^{\perp}$, the closedness of $Y_0$ and its invariance under $e^{\widehat{A}t}$.
\end{proof}

We shall utilize the fact that the compactness of $e^{\widehat{A}t}$ implies the \emph{spectrum determined growth assumption} \cite[Chapter IV, Corollary 3.12]{EngN99}, i.e., for all $\varepsilon>0$ there exists $M_\varepsilon>0$ such that
\begin{align*}
 \| e^{\widehat{A}t}\|\le M_\varepsilon e^{(\delta+\varepsilon) t}, \quad \delta := \sup\limits_{\lambda \in \sigma(\widehat{A})}( \mathrm{Re}(\lambda) ).
\end{align*}

Let us then  focus on the following linearization of \eqref{eq:op_eq}:
\begin{align*}
\ddt\widehat{y} &= \widehat{A}\widehat{y} + Bu , \quad \widehat{y}(0)=Py_0,
\end{align*}
where $B= NQy_0 = N \mathbbm{1}=-\nabla_x \alpha^\top v$ is considered as operator in $\mathcal{B}(\mathbb{R},Y_0)$.  We also recall that $\langle B,\mathbbm{1}\rangle_Y = 0$.  Below we shall denote by $B$ the operator in  $\mathcal{B}(\mathbb{R},Y_0)$ as well as the element in $Y_0$.
Further we  assume  that the initial state $y_0$ is normalized in such a way that 
\begin{align*}
 \int_{\mathbb R^{2d}} y_0 \mu \, \mathrm{d}x\, \mathrm{d}v = 1.
\end{align*}
Under Assumptions \eqref{ass:A3}-\eqref{ass:A4} the operator $\overline{\cL^\sharp_K}$ has  compact resolvent, see \cite[Theorem 3.1, pg.170,176]{HerN04}.  Consequently $A=\mathcal{M}^{-1} \overline{\cL^\sharp_K}\mathcal{M}$ has a compact resolvent as well.
Hence its spectrum consists of isolated eigenvalues $\{\lambda_j\}_{j=1}^\infty$,  with finite multiplicities,  see \cite[Theorem III.6.29]{Kat80}. The eigenvalue with largest real part is $0$. Consequently the spectrum of $\widehat A^\dagger$ consists of $\{\bar \lambda_j\}_{j=2}^\infty$ with $\mathrm{Re}( \bar \lambda_j) <0$ for all $j\ge 2$. Denote by $\psi_j$ the associated eigenfunctions.  Let $k$ be an index at which a spectral gap occurs, i.e.
 \begin{align}\label{eq:eigs}
0>\mathrm{Re}(\lambda_2)\ge \dots \ge \mathrm{Re}(\lambda_{k})= \delta > \mathrm{Re}(\lambda_{k +1} ) \ge \mathrm{Re}(\lambda_{k+2}) \dots.
 \end{align}
In case the multiplicity of the second eigenvalue of $A$ is one we can take  $k =2$ and $0>\mathrm{Re}(\lambda_2)=\delta >  \mathrm{Re}(\lambda_{3})$.

 \begin{proposition}\label{prop:hautus}
 Let Assumptions \eqref{ass:A2}-\eqref{ass:A4} hold, let $k$ be as in \eqref{eq:eigs}, and assume that $\alpha$ satisfies
 \begin{equation}\label{eq:hautus}
  \langle B,\psi_j \rangle_{Y_0} \neq 0, \text{ for  all } j\in \{2,\dots, k\}.
  \end{equation}
Then for each $\varsigma>0$ with $\delta-\varsigma > \mathrm{Re}(\lambda_{k +1})$, the pair $(\widehat{A}-\delta I +\varsigma I,B)$ is exponentially stabilizable.
\end{proposition}
\begin{proof}
 Since $(-\delta + \varsigma)I$ is a bounded perturbation of $\hat A$ the semigroup generated by $\hat A -\delta I + \varsigma I$ is compact as well, see \cite[Chapter 3, Proposition 1.4]{Paz83}, and  the infinite dimensional Hautus test is applicable for this operator, see \cite[Proposition V.1.3.3. and Remark V.1.3.5]{Benetal07}.

The spectral values which lie in the right half plane are given by $ \sigma^+ = \{\bar \lambda_j - (\delta-\varsigma)\}^{k }_{j=1}$. Thus it suffices to argue that $\mathrm{ker}(\lambda I - (\widehat{A}-(\delta  -\varsigma) I)^\dagger) \cap \mathrm{ker}(B') =\{0\}$ for each $\lambda \in \sigma^+$, or equivalently $\mathrm{ker}(\bar \lambda_j I - \widehat{A}^\dagger) \cap \mathrm{ker}(B') =\{0\} $ for each $j\in \{2,\dots, k\}$.
Note that ${B'}\in \mathcal{B}(Y_0,\mathbb{R})$ can be expressed by $  \langle B,\psi \rangle_{Y_0}$ for $\psi \in Y_0$.
Consequently if  some nontrivial  $\psi\in  \mathrm{ker}(\bar \lambda_j I - \widehat{A}^\dagger) \cap \mathrm{ker}(B')$, with $j\in \{2,\dots, k \}$,  then $\psi= \psi_j$  and $\langle B,\psi_j \rangle_{Y_0}= 0$. This contradicts \eqref{eq:hautus}, and thus $\mathrm{ker}(\lambda I - (\widehat{A}-(\delta  -\varsigma) I)^\dagger) \cap \mathrm{ker}(B') =\{0\}$, for each $\lambda \in \sigma^+$, as desired.

\end{proof}

\begin{remark}\label{rem:alpha}
\em{
 We consider the special case $k =2$ and construct a control shape function $\alpha$ such that \eqref{eq:hautus} holds.  Throughout the  discussion let Assumptions \eqref{ass:A2}-\eqref{ass:A4} hold.
By a hypoellipticity argument, see e.g. \cite[Theorem 2.13]{LelS16} it follows that $\psi_2 \in C^\infty(\mathbb{R}^{2d})$.
By direct computation utilizing \eqref{eq:aux_1} it can be shown that for
 $\alpha$ with compact support we have
\begin{equation}\label{eq:aux11}
 \langle B,\psi_2  \rangle_{Y_0} = \int _{\mathbb R^{2d}}\alpha  \nabla_x \psi_2 ^\top v \, \mu \,\mathrm{d}x\,\mathrm{d}v-
 \int _{\mathbb R^{2d}}\alpha \psi_2 \nabla_x G ^\top v \, \mu\,\mathrm{d}x\,\mathrm{d}v.
\end{equation}

 This implies that
\begin{equation}\label{eq:aux12}
 \langle B,\psi _2 \rangle_{Y_0} = \int _{\mathbb R^{d}}\alpha(x) \Psi_2(x)  e^{-G(x)}\,\mathrm{d}x,
\end{equation}
where 
\begin{equation*}
\Psi_2(x)= \int _{\mathbb R^{d}}(\nabla_x \psi_2 -\psi_2 \nabla_x G)  v \, e^{-\frac{|v|^2}{2}}\,\mathrm{d}v,
\end{equation*}
is a $W^{1,\infty}_{\mathrm{loc}}(\mathbb{R}^d)$ function. We  now assume that there exists $\rho_1>0$ such that $\Psi_2$ is not a.e. 0 in the ball $B_{\rho_1}$ with center $0$ and radius $\rho$. Further we choose $\chi\in W^{1,\infty}(\mathbb{R}^d)$ such that $\chi \ge 0$, $\chi(x)=1$ for  $x \in B_{\rho_1}$ and  $\chi(x)=0$ for $|x|>\rho_2$, where $\rho_1< \rho_2$. Then it holds
\begin{align*}
\alpha&= \chi \Psi_2 \in W^{1,\infty}(\mathbb{R}^d), \quad \alpha(x)= 0 \text{ for } |x|>\rho_2,\\
% $$
% and
% $$
 \langle B,\psi _2 \rangle_{Y_0}&= \int _{B_{\rho_2}} \chi \, \Psi^2_2(x)  e^{-G(x)}\,\mathrm{d}x >0,
\end{align*}
as desired.}
\end{remark}

\section{An infinite-horizon bilinear optimal control\\ problem}

Here we focus on an optimal stabilization problem associated to \eqref{eq:abs_bil_cau}. From \eqref{eq:split_total} it is known that the control does not effect the evolution of the state on $Y_0^{\perp}$. Therefore we concentrate on \eqref{eq:op_eq}, i.e., the controlled evolution of the state on $Y_0$. We further introduce a shift $\delta I$, with $\delta>0$, to the state equation which will guarantee an exponential decay rate for sufficiently small initial states. Setting
\begin{align}\label{eq:scaled_vars}
 \zeta(t) =e^{\delta t} \widehat{y}(t), \ \ w(t)=e^{\delta t} u(t), \ \ \widehat{N}_{\delta}(t)=e^{-\delta t} \widehat{N}
\end{align}
equation \eqref{eq:op_eq} is transformed to
\begin{align}\label{eq:scaled_dyn}
\ddt \zeta = (\widehat{A}+\delta I) \zeta + w \widehat{N}_{\delta} \zeta + w B, \ \ \zeta(0)=Py_0,
\end{align}
where we recall that $B=NQy_0=N \mathbbm{1}=-\nabla_x \alpha^\top v$. We thus consider
\begin{equation}\label{eq:opt_prob}
\begin{aligned}
  &\inf_{w\in L^2(0,\infty)} \mathcal{J}(w):=\frac{1}{2} \int_0^\infty \| \zeta(t)\|_{Y_0}^2 \, \mathrm{d}t+ \frac{\beta}{2}\int_0^\infty w(t)^2 \, \mathrm{d}t \\
  &\text{s.t. }  \ddt\zeta = (\widehat{A}+\delta I) \zeta + w \widehat{N}_{\delta} \zeta + w B  \ \ \text{in } [D(A^\dagger)\cap Y_0]', \ \  \zeta(0)=Py_0,
  \end{aligned}
\end{equation}
where $L^2(0,\infty)=L^2(0,\infty;\mathbb R)$. For convenience we point out  that the above control system is the formal notation for
\begin{equation}\label{eq:integral}
\begin{array}{ll}
\!\!\! \langle \zeta(t)-\zeta(0),z\rangle_{Y_0}  =\int_0^t\langle \zeta(s),(\widehat A^\dagger+\delta I) z \rangle_{Y_0}+ w(s) e^{-\delta s} \langle \zeta (s),\widehat {N} 'z \rangle_{Y_0} \!+ \!w(s) \langle B,z \rangle_{Y_0}\, \mathrm{d}s
 \end{array}
\end{equation}
for all $z\in \mathcal{D}(A^\dagger)\cap Y_0$. A control $w$ associated to \eqref{eq:opt_prob} is feasible  if $\mathcal{J}(w) < \infty$. In particular $\zeta \in L^2(0,\infty;Y_0)$.
Therefore, to define the solution concept associated to \eqref{eq:scaled_dyn} we can consider the $\delta \zeta$ summand as a $L^2(0,\infty;Y_0)$ perturbation, and call  $\zeta$ solution  to  \eqref{eq:scaled_dyn} if $\zeta \in L^2(0,\infty;V_v\cap Y_0)\cap L^\infty(0,\infty;Y_0)\cap C([0,\infty);Y_0)$. It is natural to include the property $\zeta \in C([0,\infty);Y_0)$, since $\widehat N$ is an admissible control operator, see the proof of Proposition  \ref{prop:sol_nonhom} below, and $t\mapsto \int_0^t e^{\widehat A(t-s)}
w(s)B \,\mathrm{d}s$ is continuous with values in $Y_0$.

\subsection{Existence of a feasible control}\label{s5.1}

Recall from, e.g., \cite[Theorem 6.2.7]{CurZ95}, that with $\alpha$ chosen according to Proposition \ref{prop:hautus}, the following operator Riccati equation has a unique stabilizing solution $\Pi \in \mathcal{B}(Y_0)$ with $\Pi=\Pi^*\succeq 0$:
\begin{equation}\label{eq:riccati}
\begin{aligned}
0=\langle \widehat{A}z_1,\Pi z_2 \rangle_{Y_0} + \langle \Pi z_1,\widehat{A}z_2 \rangle_{Y_0} + 2\delta \langle z_1,\Pi z_2 \rangle_{Y_0}  + \langle z_1,z_2\rangle _{Y_0}\!- \langle B,\Pi z_1 \rangle_{Y_0} \langle B,\Pi z_2 \rangle_{Y_0}
\end{aligned}
\end{equation}
for $z_1,z_2 \in \mathcal{D}(\widehat{A})$. In particular, it holds that the closed loop operator $A_{\pi}$ defined by
\begin{align}\label{eq:cl_op}
 A_\pi \zeta := (\widehat{A}+\delta I)\zeta - B \langle B,\Pi \zeta \rangle_{Y_0}, \quad \mathcal{D}(A_\pi)=\mathcal{D}(\widehat{A})
\end{align}
generates an exponentially stable $C_0$-semigroup $e^{A_\pi t}$ on $Y_0$. With regard to the subsequently following local fixed-point argument, let us first consider the following particular nonhomogeneous equation
\begin{align}\label{eq:non_hom}
 \ddt\zeta = A_\pi \zeta + \widehat{N} u, \ \ \zeta(0)=\zeta_0 \in Y_0.
\end{align}
We have the following result.
\begin{proposition}\label{prop:sol_nonhom} Let Assumptions \eqref{ass:A2}-\eqref{ass:A4} and \eqref{eq:hautus} hold. Then
 for all $\zeta_0 \in Y_0$ and $u\in L^2(0,\infty;Y_0)$, the unique solution of $\zeta$ of \eqref{eq:non_hom} is given by
 \begin{align}\label{eq:non_hom_mild}
  \zeta(t)= e^{A_\pi t}\zeta_0+ \int_0^t e^{A_\pi (t-s)} \widehat{N}u(s)\,\mathrm{d} s.
 \end{align}
 It holds that $\zeta \in L^2(0,\infty;V_v \cap Y_0) \cap L^\infty(0,\infty;Y_0)\cap C([0,\infty);Y_0)$. Moreover, there exists a constant $M>0$ s.t.
 \begin{align}\label{eq:non_hom_aprio}
 \mathrm{max}(  \| \zeta \|_{L^2(0,\infty;V_v)},\| \zeta \|_{L^\infty(0,\infty;Y)}) \le M ( \| \zeta_0 \| _{Y_0} + \| u\|_{L^2(0,\infty;Y_0)}).
 \end{align}
\end{proposition}
\begin{proof}
In the proof of Proposition \ref{prop:mild_solution_lin}, we have already shown that $N$ is an (infinite-time) admissible control operator for $e^{At}$. From \eqref{eq:aux8} we conclude that also $\widehat{N}$ is an (infinite-time) admissible control operator for $e^{\widehat{A}t}$. We further have that $\widehat{N}$ is an admissible control operator for $e^{(\widehat{A}+\delta I )t}$. Now recall that $A_\pi$ is defined by
\begin{align*}
 A_\pi = \widehat{A}+\delta I - B B' \Pi = \widehat{A}+\delta I - N \mathbbm{1} B'\Pi= \widehat{A}+\delta I - \widehat{N} \mathbbm{1} B'\Pi.
\end{align*}
Since $D:=\mathbbm{1} B'\Pi\in \mathcal{B}(Y_0)$, with \cite[Corollary 5.5.1]{TucW09}, we obtain that $\widehat{N}$ is an admissible control operator for $e^{A_\pi t}$. Since $e^{A_\pi t}$ is exponentially stable, from \cite[Proposition 4.4.5]{TucW09}, it additionally follows that $N$ is infinite-time admissible. Altogether, this implies the existence of a constant $M_1$ such that
\begin{align*}
 \sup\limits_{t\ge 0} \| \int_0^t e^{A_\pi (t-s)} \widehat{N}u(s)\, \mathrm{d}s\|_{Y_0}\le M_1 \| u\|_{L^2(0,\infty;Y_0)} \ \ \forall u \in L^2(0,\infty;Y_0).
\end{align*}
This estimate together with the exponential stability of $e^{A_\pi t}$ implies the second estimate in \eqref{eq:non_hom_aprio}. With regard to the first estimate, let us first note that the mapping
\begin{align}\label{eq:in_out_map}
F\colon L^2_{\mathrm{loc}}(0,\infty;Y_0)\to L^2_{\mathrm{loc}}(0,\infty;Y_0), \ \ u\mapsto Fu = \zeta,
\end{align}
where $\zeta$ is defined by \eqref{eq:non_hom_mild} is the \emph{input-output map} associated with the trivial observation operator $C=I$, see, e.g., \cite{TucW14}. In particular, the system described by $(A_\pi,\widehat{N},I)$ is a well-posed linear system in the sense of \cite{Sta05,TucW14}. Again using the exponential stability of $e^{A_\pi t}$, we conclude that the growth bound of $F$ is strictly negative and, hence, $F\in \mathcal{B}(L^2(0,\infty;Y_0))$, see, e.g., \cite[Proposition 4.7]{TucW14} or \cite[Theorem 2.5.4(iii)]{Sta05}. Thus, there exists a constant $M_2>0$ such that
\begin{align}\label{eq:l2_y_est}
  \| \zeta \| _{L^2(0,\infty;Y_0)} \le M_2 \left( \| \zeta_0\|_{Y_0} + \| u\| _{L^2(0,\infty;Y)} \right).
  \end{align}
For the additional (spatial) regularity of $\zeta$, note that we can resort to the linear system
\begin{align*}
 \ddt\zeta = \widehat{A}\zeta + f,
\end{align*}
where $f:= \delta \zeta - B\langle B,\Pi \zeta \rangle_{Y_0} + \widehat{N}u$ satisfies $f \in L^2(0,\infty;V_v')$. With \eqref{eq:aprio_lin} in Proposition \ref{prop:lin_more_reg}, we can argue that there exists a constant $M_3$ such that
\begin{align}\label{eq:l2_v_est}
 \|\nabla_v \zeta \| _{L^2(0,\infty;Y)} \le M_3 \left( \| \zeta_0\|_{Y_0} +  \| u\| _{L^2(0,\infty;Y)} \right).
\end{align}
Combining \eqref{eq:l2_v_est} and \eqref{eq:l2_y_est} shows the announced estimates. Moreover, by \cite[Proposition 4.2.5]{TucW09}, the mild solution $\zeta$ in \eqref{eq:non_hom_mild} is continuous.
\end{proof}

\begin{lemma}\label{lem:loc_lip}
 Let $\zeta_1,\zeta_2 \in L^2(0,\infty;Y_0) \cap L^\infty(0,\infty;Y_0)$. Then there exists a constant $\widetilde{M}>0$ such that
 \begin{align*}
  &\|e^{-\delta \cdot}( \langle B,\Pi \zeta_1 \rangle _{Y_0}  \zeta_1 -
  \langle B,\Pi \zeta_2 \rangle _{Y_0}  \zeta_2) \|_{L^2(0,\infty;Y)}\\
  &\quad\le \widetilde{M}\left( \| \zeta_1\|_{L^\infty(0,\infty;Y)} + \| \zeta_2\|_{L^\infty(0,\infty;Y)}   \right)\| \zeta_1-\zeta_2\| _{L^2(0,\infty;Y)}.
 \end{align*}
\end{lemma}
\begin{proof}  
\begin{align*}
& \|e^{-\delta \cdot}( \langle B,\Pi \zeta_1 \rangle   \zeta_1 -
  \langle B,\Pi \zeta_2 \rangle_{Y_0}   \zeta_2) \|_{L^2(0,\infty;Y)} \\
  & \le \|e^{-\delta \cdot} \langle B,\Pi \zeta_1 \rangle_{Y_0} (  \zeta_1 - \zeta_2)\|_{L^2(0,\infty;Y)} + \| e^{-\delta \cdot}
  (\langle B,\Pi \zeta_1 \rangle_{Y_0}-\langle B,\Pi \zeta_2 \rangle_{Y_0})   \zeta_2) \|_{L^2(0,\infty;Y)} \\
  & \le \| e^{-\delta \cdot } \langle B,\Pi \zeta_1 \rangle \|_{L^\infty(0,\infty)} \| \zeta_1 - \zeta_2 \| _{L^2(0,\infty;Y)} \\
  &\qquad + \| e^{-\delta \cdot} \zeta_2 \|_{L^\infty(0,\infty;Y_0)} \| \langle \Pi B,\zeta_1-\zeta_2 \rangle_{Y_0} \| _{L^2(0,\infty)} \\
  & \le\widetilde{M}( \| \zeta_1 \| _{L^\infty(0,\infty;Y)}+\| \zeta_2 \| _{L^\infty(0,\infty;Y)}) \| \zeta_1 - \zeta_2 \| _{L^2(0,\infty;Y)},
\end{align*}
for a constant $\widetilde{M}$ depending on $\|B\|_{Y_0}$ and $\|\Pi\|_{\mathcal{B}(Y_0)}$.
\end{proof}

\begin{theorem}\label{theo:feasible}Let Assumptions \eqref{ass:A2}-\eqref{ass:A4} and \eqref{eq:hautus} hold. Let $M,\widetilde{M}$ denote the constants from Proposition \ref{prop:sol_nonhom} and Lemma \ref{lem:loc_lip}, respectively. If $\| \zeta_0\| < \frac{3}{16M^2 \widetilde{M}},$ then 
 \begin{align*}
  \ddt\zeta = A_\pi \zeta - \langle B,\Pi \zeta \rangle_{Y_0} \widehat{N}_{\delta} \zeta, \ \ \zeta(0)=Py_0,
 \end{align*}
admits a unique solution $\zeta \in L^2(0,\infty;V_v\cap Y_0)\cap L^\infty(0,\infty;Y_0)\cap C([0,\infty);Y_0)$. This solution satisfies \\[-4ex]
\begin{align*}
\max( \| \zeta \|_{L^2(0,\infty;V_v)},\| \zeta\|_{L^\infty(0,\infty;Y_0)}) \le \tfrac{1}{4 M \widetilde{M}}.
\end{align*}
\end{theorem}
\begin{proof}
  This result can be proved by arguments similar to those provided in the proofs of \cite[Theorem 4.8]{BreKP17} and \cite[Theorem 25]{BreKP19}.
\end{proof}

Theorem  \ref{theo:feasible} implies that the choice $w=- \langle B,\Pi \zeta \rangle_{Y_0}\in L^2(0,\infty)$ is a feasible control for \eqref{eq:opt_prob}. In particular, the associated state satisfies \eqref{eq:integral}.

\subsection{Existence of an optimal control}\label{s5.2}

Here we provide sufficient conditions which guarantee the existence of an optimal control for problem \eqref{eq:opt_prob}.

\begin{theorem}\label{theo:existopt}
Let Assumptions \eqref{ass:A2}-\eqref{ass:A4} hold and assume the existence of an admissible control for \eqref{eq:opt_prob}.  Then \eqref{eq:opt_prob} admits an optimal control $\bar w \in L^2(0,\infty)$.
\end{theorem}
\begin{proof}

Since $\mathcal{J}$ is bounded from below by $0$, and due to the assumption on the existence of an admissible control, there exists a minimizing sequence $\{w_n\}_{n=1}^\infty$ for \eqref{eq:opt_prob}. Observe that $\zeta_n=\zeta(w_n)$ satisfy
\begin{equation}\label{eq:approx}
\begin{array}{ll}
 \langle \zeta_n(t)-\zeta(0),z\rangle_{Y_0} \\[1.5ex]
 =\int_0^t\langle  \zeta_n(s),(\widehat A^\dagger+\delta I) z \rangle_{Y_0}+ w_n(s) \langle  \zeta_n (s), \widehat{N}_\delta(s)' z \rangle_{Y_0} +  w_n(s) \langle B,z \rangle_{Y_0}\, \mathrm{d}s
 \end{array}
\end{equation}
for all $z\in \mathcal{D}( A^\dagger)\cap Y_0$, and each  $t>0$. Next we pass to the limit  $n\to \infty$ in the above equation.  The choice of the cost functional implies that $\{w_n\}_{n=1}^\infty$ is bounded in $L^2(0,\infty)$. Let us now fix an arbitrary $t>0$. We can  follow the proof of Theorem \ref{thm:global_mild_solution} to assert that $\{\zeta_n\}_{n=1}^\infty$  is bounded in $C([0,t]; Y_0)$. Using \eqref{eq:scaled_dyn}, the boundedness of  $\{\zeta_n\}_{n=1}^\infty$  in $W^{1,2}(0,t; [ \mathcal{D}( A^\dagger)]')$ follows. Thus, for a subsequence, denoted by the same symbols, we have that
\begin{equation*}
w_n\rightharpoonup \bar w \text{ in } L^2(0,\infty),\, \zeta_n\rightharpoonup \bar \zeta \text{ in } W^{1,2}(0,t;[ \mathcal{D}( A^\dagger)]'), \text{and } \zeta_n \stackrel{w^*}{\rightharpoonup} \bar \zeta \text{ in } L^\infty(0,\infty; Y_0),
\end{equation*}
for some $\bar w \in L^2(0,\infty)$, and $\bar \zeta_{|(0,t)} \in W^{1,2}(0,t;[ \mathcal{D}( A^\dagger)]' )\cap $ and $\bar{\zeta} \in  L^\infty(0,\infty;Y_0)$. This implies that $\langle \zeta_n,z\rangle_{\mathcal{D}} \rightharpoonup   \langle \bar{\zeta},z\rangle_{\mathcal{D}}$ in $W^{1,2}(0,t)$, i.e.
\begin{equation*}
\langle \zeta_n,z\rangle_{\mathcal{D}}\rightharpoonup \langle \bar \zeta,z\rangle_{\mathcal{D}}, \text{ and } \langle \ddt \zeta_n,z\rangle_{\mathcal{D}}\rightharpoonup \langle \ddt \bar\zeta,z\rangle_{\mathcal{D}} \text{ in } L^2(0,t).
\end{equation*}
Due to the compact embedding $W^{1,2}(0,t)\subset C([0,t])$, it follows that
\begin{equation}\label{eq:aux13}
\langle \zeta_n,z\rangle_{\mathcal{D}}\to \langle \bar \zeta,z\rangle_{\mathcal{D}} \text{ in } C([0,t])  \text{ for all } z\in \mathcal{D}( A^\dagger) \cap Y_0.
\end{equation}
Consequently, we can pass to the limit on the left hand side of \eqref{eq:approx} to obtain 
\begin{align}\label{eq:aux13a}
 \langle \zeta_n(s)-\zeta(0),z\rangle_{Y_0} \to \langle \bar{\zeta} - \zeta(0),z\rangle_{Y_0}
\end{align}
for all $s\in (0,t)$.

% Arguing as above \eqref{eq:eigs} and using again \cite[Theorem 3.1, pg.170, 176]{HerN04} the operator $\widehat A^\dagger$ has a compact resolvent and hence $\mathcal{D}( A^\dagger) \cap Y_0$ is compact in $Y_0$.
Using the density of $\mathcal{D}(A^\dagger)\cap Y_0$ in $Y_0$, a three-$\epsilon$ argument and \eqref{eq:aux13} we find 
\begin{equation*}
\langle \zeta_n,g\rangle_{Y_0}\to \langle \bar \zeta,g\rangle_{Y_0} \text{ in } L^2(0,t)  \text{ for all } g\in Y_0.
\end{equation*}
Since $\widehat N' \in \mathcal{B}(\mathcal{D}(A^\dagger)\cap Y_0,Y_0)$, this implies that 
\begin{align}\label{eq:aux14}
\langle \zeta_n(\cdot),\widehat{N}_{\delta}(\cdot)'z\rangle _{Y_0} \to
\langle \zeta_n(\cdot),\widehat{N}_{\delta}(\cdot)'z\rangle _{Y_0} 
\end{align}
in $L^2(0,s)$ for all $ s \in (0,t)$.
Using \eqref{eq:aux13a}, \eqref{eq:aux14} and the fact that $\widehat N' \in \mathcal{B}(\mathcal{D}(A^\dagger)\cap Y_0,Y_0)$ we can  pass to the limit in \eqref{eq:approx} to obtain for all $z\in \mathcal{D}(A^\dagger)\cap Y_0$ and every $s\in [0,t]$:
\begin{equation}\label{eq:limit_solution}
\begin{array}{ll}
 &\langle \bar \zeta(s)\!-\!\zeta(0),z\rangle_{Y_0}
 \! \\
 &\qquad =\!\int_0^s\big(\langle \bar \zeta(\sigma),(\widehat A^\dagger+\delta I) z \rangle_{Y_0}\!+\! \bar w(\sigma) \langle  \bar \zeta (\sigma), \widehat{N}_\delta(\sigma)' z \rangle_{Y_0} \!+  \!\bar w(\sigma) \langle B,z \rangle_{Y_0}\, \big)\mathrm{d}\sigma.
 \end{array}
\end{equation}
 Thus $\bar \zeta$ is the solution associated to the control $\bar w$. To assert that this identity is true for all $s\in [0,\infty)$, we argue by extension. Let $\tilde{t}>t$, then following the above arguments we obtain for all $s\in [0,\tilde{t}]$ that
\begin{equation*}
\begin{array}{ll}
 &\langle \bar \zeta_{\tilde{t}}(s)\!-\!\zeta(0),z\rangle_{Y_0}
 \! \\
 &\qquad =\!\int_0^s\big(\langle \bar \zeta_{\tilde{t}}(\sigma),(\widehat A^\dagger+\delta I) z \rangle_{Y_0}\!+\! \bar w(\sigma) \langle  \bar \zeta_{\tilde{t}} (\sigma), \widehat{N}_\delta(\sigma)' z \rangle_{Y_0} \!+  \!\bar w(\sigma) \langle B,z \rangle_{Y_0}\, \big)\mathrm{d}\sigma.
 \end{array}
 \end{equation*}
 Since for fixed control $\bar{w}$, the bilinear system has a unique solution, we conclude that $(\bar{\zeta}_{\tilde{t}})_{|(0,t)}=\bar{\zeta}$. Consequently, \eqref{eq:limit_solution} holds for all $s\in [0,\infty)$. 
 
 Let us further note that since $\bar{w}\in L^2(0,\infty)$ and $\bar{\zeta}\in L^\infty(0,\infty;Y_0)$, we can use that $t\mapsto \int_0^t e^{\widehat{A}(t-s)}\bar{w}(s) B\,\mathrm{d} s$ is continuous with values in $Y_0$ together with Proposition \ref{prop:sol_nonhom} applied for $u=\bar{w}\bar{\zeta}$ to obtain $\bar{\zeta}\in C([0,\infty);Y_0)$.
 
Finally, by weak lower semi-continuity of norms, we have that
\begin{equation*}
\mathcal{J}(\bar w)= \frac{1}{2} \int_0^\infty \|\bar \zeta(t)\| _{Y_0}^2 \, \mathrm{d}t+ \frac{\beta}{2}\int_0^\infty \bar w(t)^2 \, \mathrm{d}t \le \liminf_{n\to\infty} \mathcal{J}(w_n),
\end{equation*}
and thus $\bar w$ is an optimal solution for \eqref{eq:opt_prob}.

\end{proof}

\section{Numerical experiments}\label{s6}

Here we complement our theoretical results with two numerical examples. These  should be considered as a proof  of concept rather than a complete  numerical investigation which, while certainly of interest, is out of the scope of the current manuscript. Indeed, the challenges include handling the unboundedness of the spatial domain and the fact that the infinite-horizon optimal control problem is posed for an unstable system. Note also that $\mu$ scales exponentially such that the numerical realization of the weighted inner products may suffer from finite numerical precision.

The controls we implemented correspond to the Riccati-based strategy discussed in the context of feasibility of \eqref{eq:opt_prob} in Section \ref{s5.1}. While such controls are not optimal, their performance is often sufficient for practical purposes. In particular this holds true for small perturbations around the steady state which is a consequence of the feedback law $\langle B,\Pi \zeta \rangle_{Y_0}$ being a second order Taylor approximation of the optimal feedback law obtained by differentiating the minimal value function, see \cite{BreKP19,TheBR10}.

All simulations were generated on an Intel i5-9400F @ 4.1 GHz x 6,
64 GB RAM,  \matlab \;version R2019b. For the solutions of the nonlinear ODE systems, we utilize the  \matlab\;routine \texttt{ode23} with (default) relative and absolute tolerances.

\subsection{Spatial discretization and numerical realization}\label{s6.1}

For the numerical realization of the (un)controlled systems \eqref{eq:Psub} and \eqref{eq:op_eq}, we replace the infinite-dimensional systems by spatially discrete surrogate models.

 For the uncontrolled potential $G$, we choose the   triple well potential discussed in \cite{BreKP17} and defined by
 $
% \begin{align*}
 G(x):=\tfrac{ ((\tfrac{1}{2}x^2 - 15)x^2 + 119)x^2 + 28x + 50}{200}.
% \end{align*}
$
As the underlying computational domain is $\mathbb R^{2}$, we utilize a spectral method based on ``Sinc'' cardinal functions for which we briefly recall the presentation in \cite[Section 5.2]{Boy00}. Spatial approximations are assumed to be of the form
\begin{align*}
 f(x)\approx f_N(x) = \sum_{j=-N}^{N} f(x_j) C_j(x), \ \ C_j(x) = \frac{\mathrm{sin}(\pi (x-jh)/h)}{\pi(x-jh)/h}
\end{align*}
where the spectral collocation points $x_j=jh,j=-N,\dots,0,\dots,N$ are uniformly distributed. Note that the functions $C_j$ are also known as Whittaker's cardinal functions and satisfy $C_k(x_j)=\delta_{jk}$. In particular, let us emphasize the beneficial discrete structure which leads to exact (skew-)symmetric finite difference approximations of the first and second derivatives \cite[Appendix F.7]{Boy00}. On the downside, the resulting matrices are generically dense and thereby limit the maximal number of spatial degrees of freedom to a relatively small dimension. As a consequence, we report on results obtained on a two-dimensional grid $(-5,5)^2$ with $2N+1=161$ grid points of grid size $h=0.0625$ in each spatial direction, leading to a spatial approximation $A \in \mathbb R^{25921\times 25921}$ of the operator $\overline{\mathcal{L}^*}$. For the decoupling of the invariant measure, we evaluate the analytic expression $\mathrm{exp}(-H(x,v))$ on the computational grid and compute a projection on the $25920$-dimensional stable subspace according to the strategy described in \cite[Section 4.1]{BreKP17}.
Rather than taking a single input as in  \eqref{eq:langevincontrol}, we chose four control potential functions $\alpha_1,\alpha_2,\alpha_3,\alpha_4$ with $\alpha_i(x)=(\exp(-(x+(2i-5)))+1)^{-1}$. This multi-input configuration led to more robust stabilization and solution results for the underlying algebraic Riccati equations. The associated spatial derivatives which enter \eqref{eq:langevincontrol} are shown in Figure \ref{fig:control} (left). Clearly the theoretical results can be extended to this multi-input case.

For the guaranteed exponential decay rate, we choose $\delta = 0.2$. The associated (shifted) algebraic Riccati equation \eqref{eq:riccati} is approximately solved for $\Pi$ by a Kleinman-Newton iteration \cite{Kle68}. Since the discrete approximation of $\widehat{A}+0.2 I$ has two unstable eigenvalues, see Figure \ref{fig:spectrum} on the right, we compute a stabilizing initial approximation $\Pi_0$ by a spectral projection onto the subspace generated by the first two eigenvectors and a subsequent solution of the corresponding two-dimensional Riccati equation. The computation of iterates $\Pi_0,\Pi_1,\dots,\Pi_k$ is stopped once $\|\Pi_k-\Pi_{k-1}\|< 10^{-5}$. Each iterate $\Pi_k$ is obtained via the \matlab\;routine \texttt{lyap} which solves the Lyapunov equation associated with the current feedback system given by $A_k=A+\delta I -BB^\top \Pi_k$.

\subsection{Examples}\label{s6.2}

We present exponential stabilization results for the two initial configurations shown in Figure \ref{fig:initial_conf} (center/right). The first configuration corresponds to a smooth perturbation of the stationary distribution. In more detail, instead of the constant function $\mathbbm{1}(x,v)\equiv1$, denoting the coordinates of the desired stationary state $\mu(x,v)$ (in the weighted space), we defined an initial state of the form $y_0(x,v)=1+\tfrac{1}{2}\cos(2\pi x) \sin(\tfrac{1}{2}\pi v)$. The resulting initial configuration (in the unweighted state space) is presented in Figure \ref{fig:initial_conf} (left).

 The second configuration is obtained as weighted average of the stationary distribution and a 90 degree rotation thereof.

\begin{figure}[htbp]
  \begin{subfigure}[c]{0.32\textwidth}
    \includegraphics[scale=0.28]{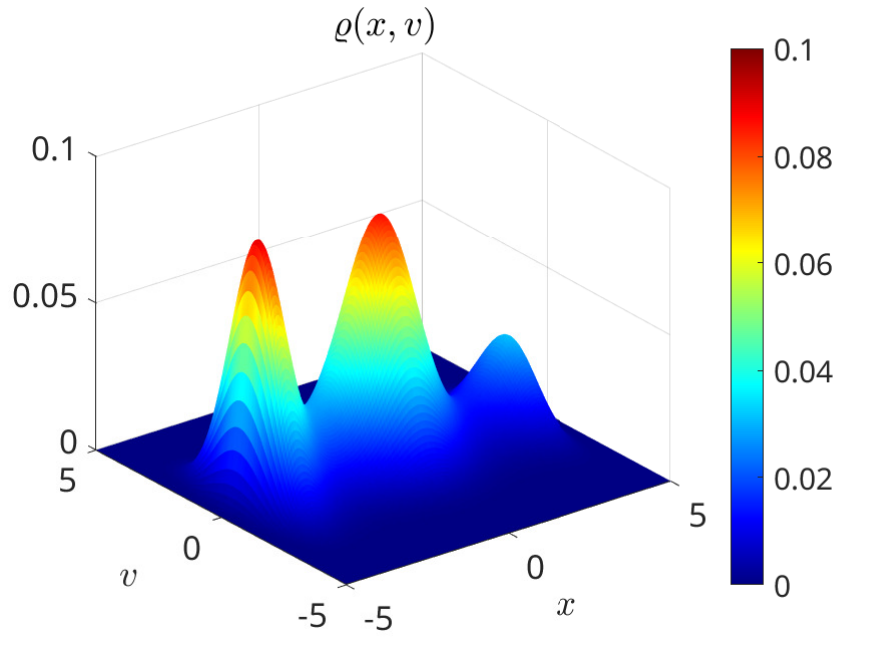}
    \end{subfigure}\ \begin{subfigure}[c]{0.32\textwidth}
    \includegraphics[scale=0.28]{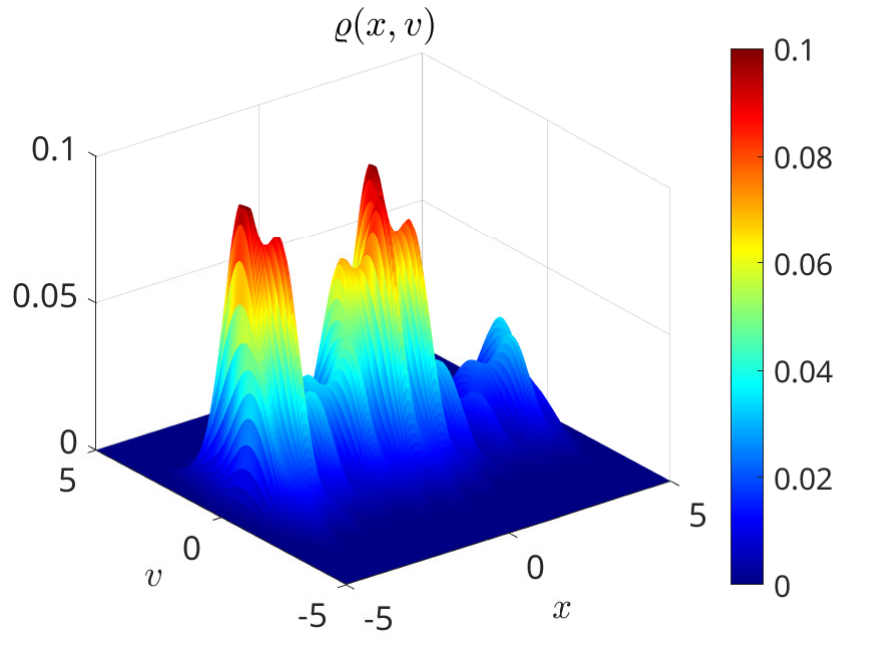}
    \end{subfigure}\ \begin{subfigure}[c]{0.32\textwidth}
    \includegraphics[scale=0.28]{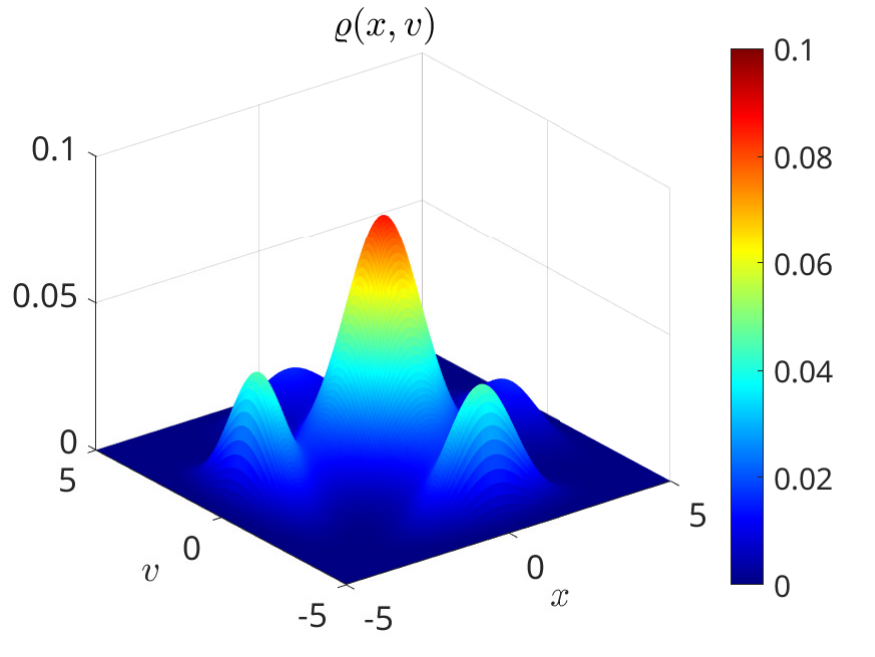}
    \end{subfigure}
    \caption{\emph{Left.} Stationary distribution for a confining triple well potential. \emph{Center/Right.} Two different initial configurations.}
  \label{fig:initial_conf}
    \end{figure}

In Figure \ref{fig:trajectories}, we show the exponential decay of the trajectories to the equilibrium state defined by $\varrho_\infty(x,v)=\mu(x,v)$. With regard to the initial configuration defining a  perturbation of $\mu$, see Figure \ref{fig:trajectories} (left), we observe almost identical behavior of uncontrolled and controlled dynamical states in the beginning of the time interval. At $t\approx 5$, the local effect of the linearized dynamics dominates and causes the expected improved exponential convergence rate. As is evident from Figure \ref{fig:control} (center), the influence of the feedback controls is small. Interestingly enough, these control laws still lead to a clear improvement in the speed of convergence to the equilibrium.

    \begin{figure}[htbp]
  \begin{subfigure}[c]{0.49\textwidth}
    \includegraphics[scale=0.43]{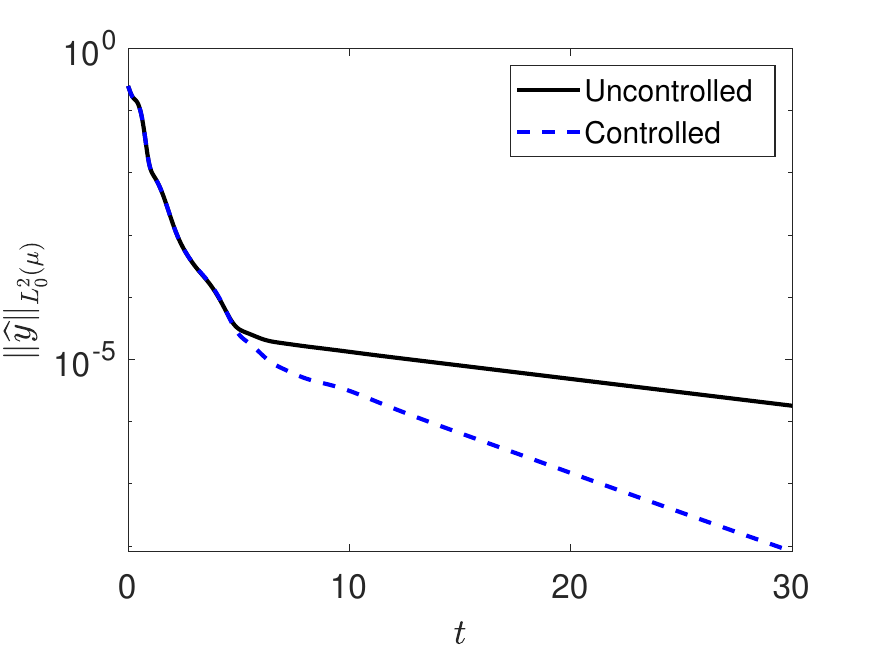}
    \end{subfigure}\ \begin{subfigure}[c]{0.49\textwidth}
    \includegraphics[scale=0.43]{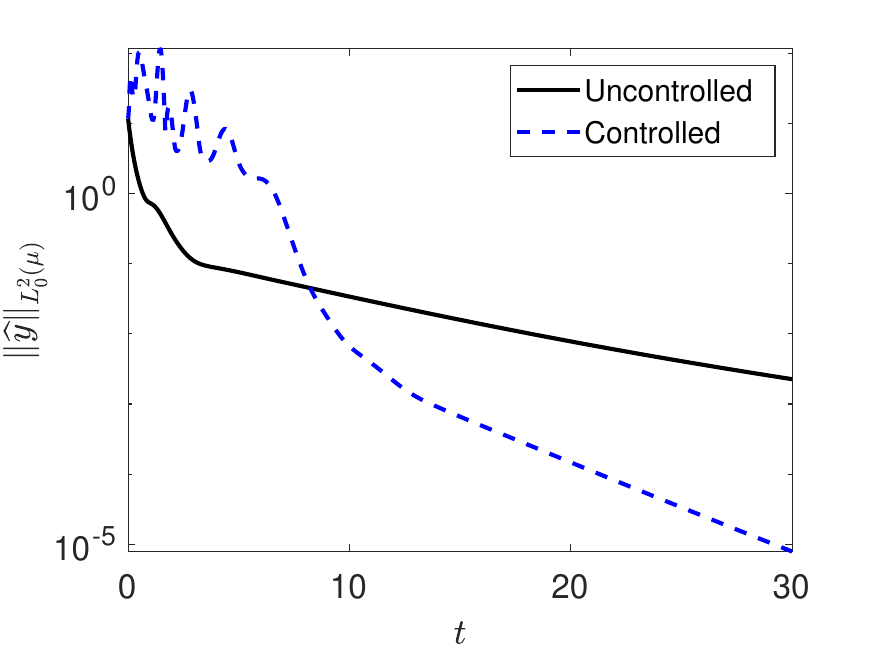}
    \end{subfigure}
    \caption{Exponential decay of the (un)controlled trajectories for perturbed (\emph{left}) and rotated (\emph{right}) initial configurations.}
  \label{fig:trajectories}
    \end{figure}

  For the initial configuration involving a rotation of the stationary distribution, Figure \ref{fig:trajectories} (right) shows that the nonlinear closed loop system requires some time to convergence into a suitable neighborhood of $\mu$ within which the prescribed exponential decay rate is obtained. Currently, we do not know whether the oscillatory behavior for $t\in [0,5]$ is an accurate approximation of the true dynamics or rather a numerical artifact caused by a not sufficiently refined discrete grid. In this regard, in Figure \ref{fig:spectrum} (left) we further provide the spectral properties of the uncontrolled and controlled (linearized) dynamics, i.e., the eigenvalues of $A$ and $\widehat{A}-BB^\top \Pi$, respectively. With regard to Remark \ref{rem:pH}, particularly note that the presence of a skew symmetric part $J$ renders the spectrum complex-valued. Moreover, the difference of both spectra mainly becomes noticeable around the unstable region of $A+\delta I$, see Figure \ref{fig:spectrum} (right). By construction, we observe that $\mathrm{Re} (\lambda_i(\widehat{A}-BB^\top \Pi))<\delta$ for all $i$.

    \begin{figure}[htbp]
  \begin{subfigure}[c]{0.475\textwidth}
    \includegraphics[scale=0.4]{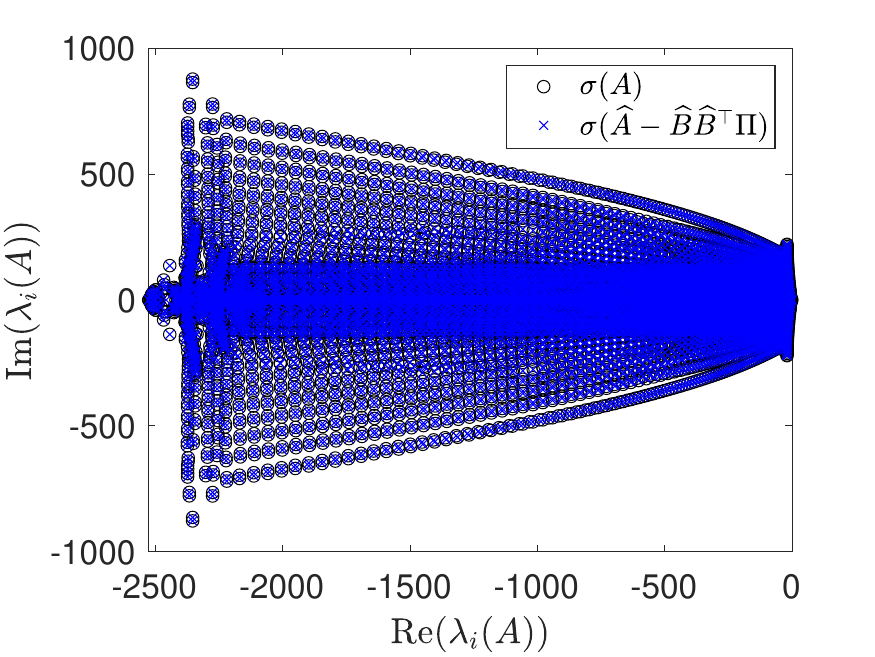}
    \end{subfigure}\ \begin{subfigure}[c]{0.475\textwidth}
    \includegraphics[scale=0.4]{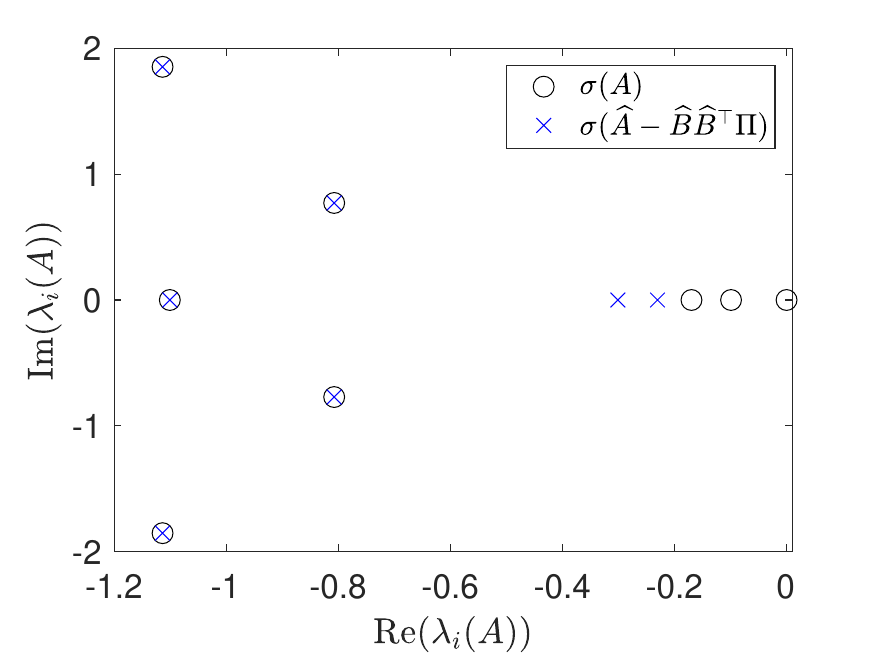}
    \end{subfigure}
    \caption{\emph{Left.} Entire spectrum of the (un)controlled system. \emph{Right}. Zoom around the $\delta$-unstable part.}
  \label{fig:spectrum}
    \end{figure}

  In comparison to the perturbed initial configuration, the influence of the feedback control laws is  significantly stronger, see Figure \ref{fig:control} (right). The oscillatory behavior of the controls may require additional numerical investigation.

    \begin{figure}[htbp]
  \begin{subfigure}[c]{0.32\textwidth}
    \includegraphics[scale=0.28]{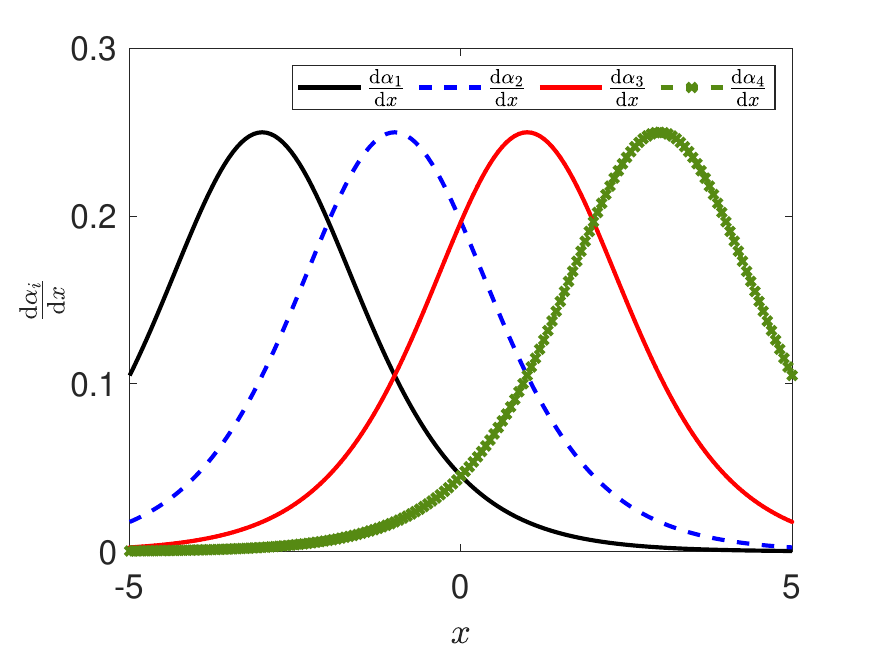}
    \end{subfigure}\ \begin{subfigure}[c]{0.32\textwidth}
    \includegraphics[scale=0.28]{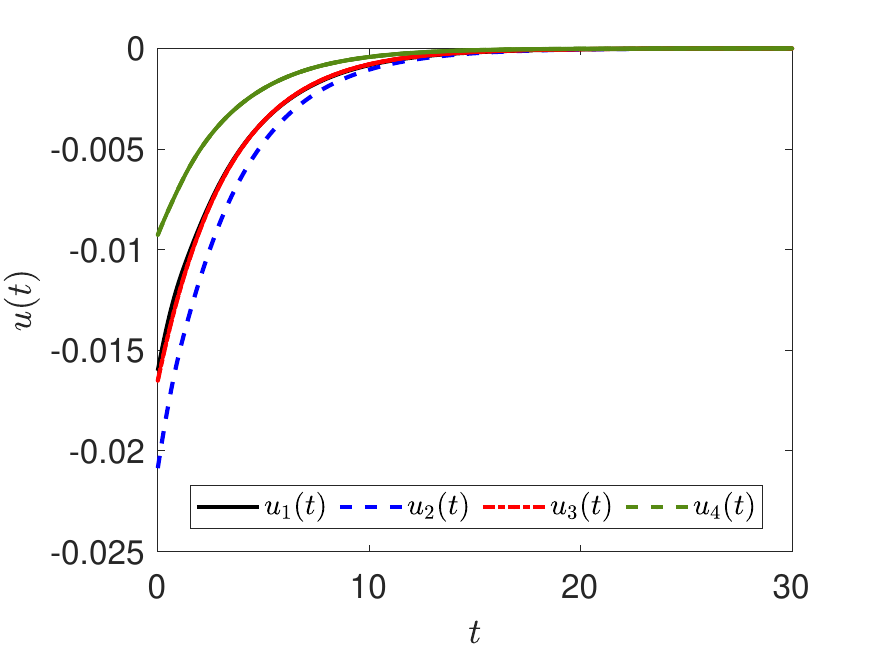}
    \end{subfigure}\ \begin{subfigure}[c]{0.32\textwidth}
    \includegraphics[scale=0.28]{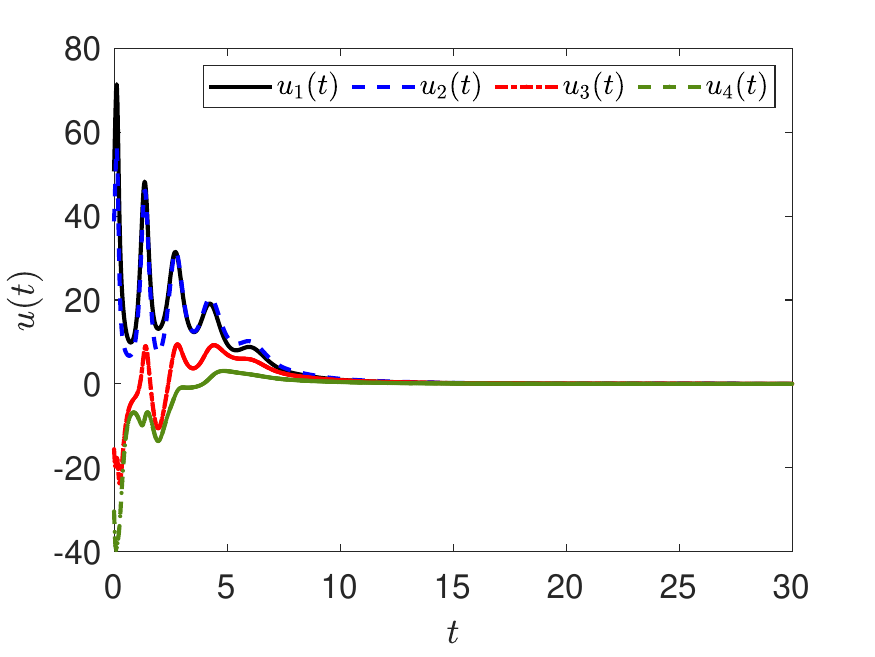}
    \end{subfigure}
    \caption{\emph{Left.} Gradients of the control shape functions $\alpha_i$. \emph{Center.} Feedback control laws for the perturbed initial configuration. \emph{Right.} Feedback control laws for the rotated initial configuration.}
  \label{fig:control}
    \end{figure}

\section*{Acknowledgement}

T.~Breiten acknowledges funding from the DFG
 via the MATH+ Cluster of Excellence, EXC 2046, in particular via the subproject EF4-6. We want to thank F.~Schwenninger (U Twente) for helpful discussions on admissibility and well-posed linear systems.

 \appendix

\section{Proofs}\label{appendix:proofs}

% \begin{proof}[Proof of \cref{le:density}]
%  In the case without weights this fact is well-known. For the weighted space $V$, the assertion can be verified by adapting the proof of \cite[Theorem 3.18]{Ada75} and we only sketch the most important steps. For fixed but arbitrary $u\in V$, we follow \cite[Theorem 3.18]{Ada75} and first define $f_{\varepsilon}$ and $u_\varepsilon=f_{\varepsilon} u$ (with compact support) such that
% \begin{align*}
%  \| \nabla u_{\varepsilon}(x,v)\| \le K\| \nabla u(x,v)\| ,\quad \| u_{\varepsilon}(x,v)\| \le K \| u(x,v)\|,
% \end{align*}
% with $K$ independent of $(x,v)$ and of $u$. With  $\Omega_{\varepsilon}=\{(x,v) \ | \ \|(x,v)\| > \tfrac{1}{\varepsilon} \}$  and $u=u_\varepsilon$ on $\Omega_{\varepsilon}^c$ we can show that $\| u-u_{\varepsilon}\|_V^2 \to 0$ for $\varepsilon \to 0$. For fixed $\varepsilon$, let us then introduce an open, bounded neighborhood $U_0$ containing $\overline{K}$ where $K=\{ (x,v)  \ |\  u_\varepsilon(x,v) \neq 0\}$ and apply \cite[Lemma 3.15]{Ada75} which provides us with $u_\eta \in C^{\infty}(\mathbb R^{2d})$ with compact support in $\overline{U}_0$ such that
% \begin{align*}
%  &\| u_{\varepsilon}-u_{\eta} \| _{W^{1,2}(\mathbb R^{2d})} < \eta, \ \
%  \| u_{\varepsilon}-u_{\eta} \| _V^2 \le e^{-\mathrm{min}_x G(x)}\eta^2.
% \end{align*}
% This allows to show that $\|u-u_{\varepsilon}\|_ V \le \eta (1+ e^{-\tfrac{1}{2}\min_x G(x)})$ where we choose $\eta$ sufficiently small to obtain the announced statement.
% \end{proof}

\begin{proof}[Proof of \cref{thm:local_mild_solution}]
 We use techniques from  \cite[Proposition 2.1]{BalMS04}. For $R>0$ let us define the set
 \begin{align*}
   \mathcal{F}= \left\{ y \in C([0,t_0];Y) \ | \ \| y-y_0 \|_{L^\infty(0,t_0;Y)} \le R \right\},
 \end{align*}
 where $t_0$ is to be determined below, and the mapping $\mathcal{T}_u \colon \mathcal{F}\to C([0,t_0];Y)$ by
\begin{align*}
  (\mathcal{T}_u y)(t) =e^{At} y_0 + \int _0^t e^{A(t-s)} Ny(s)u(s)\, \mathrm{d}s .
\end{align*}
We first  show that $\mathcal{T}_u$ maps $\mathcal{F}$ into itself. For all $0\le t\le t_0$ we estimate
\begin{align*}
\| (\mathcal{T}_u y)(t)- y_0 \|_Y \le
 \| e^{At}y_0 - y_0\| _Y + \| \Phi_t (yu) \|_Y
\end{align*}
where $\Phi_t$ is the controllability map defined in \eqref{eq:cont_map}.
Since $N$ is an (infinite-time) admissible control operator for $e^{At}$, \cite[Remark 4.6.2]{TucW09} yields the existence of a constant $M$ (independent of $t$) such that for $z(t)=\Phi_t (yu)$ we have
\begin{align*}
 \| z\|_{L^\infty(0,t_0;Y)} \le M \| yu\|_{L^2(0,t_0;Y)} \le M \|y\| _{L^\infty(0,t_0;Y)} \| u\| _{L^2(0,t_0)}.
\end{align*}
This implies for $0\le t\le t_0$:
\begin{align*}
 \| (\mathcal{T}_u y)(t)- y_0 \|_Y \le
 \| e^{At}y_0 - y_0\| _Y + M (R+\|y_0\|_Y) \| u\|_{L^2(0,t_0)}.
\end{align*}
Due to strong continuity of $e^{At}$ and the fact that $\| u\|_{L^2(0,t)}\to 0 $ for $t\to 0$, there exists $t_0$ such that
\begin{align*}
 \| (\mathcal{T}_u y)(t)- y_0 \|_Y \le  R \quad \text{ for all $0\le t\le t_0$}.
\end{align*}
Similarly, considering two solutions $y$ and $\tilde{y}$ with $y(0)=y_0=\tilde{y}(0)$, we find $t_0$ and $c\in (0,1)$ such that for all $0\le t\le t_0$:
\begin{align*}
 &\| (\mathcal{T}_u y)(t)- (\mathcal{T}_u \tilde{y})(t)\| _Y = \| \Phi_t (y-\tilde{y})u\| _Y \\& \qquad \le M \| u\|_{L^2(0,t_0)} \| y-\tilde{y}\| _{L^\infty(0,t_0;Y)}  \le c \| y-\tilde{y}\| _{L^\infty(0,t_0;Y).}
\end{align*}
Now, the assertion follows by a standard fixed point argument.
\end{proof}

\begin{proof}[Proof of \cref{thm:global_mild_solution}]
 Again, we follow the line of argument provided in \cite[Theorem 2.5]{BalMS04} for an analogue statement for bounded control operators. We also utilize ideas from the proof of \cite[Theorem 2.9]{HosJS18}. Thus, suppose that $y(\cdot)$ solves \eqref{eq:abs_bil_cau} defined for $t\in [0,a)$ with $a\le T$. Then for a constant $c$ and the constant $M$ from the proof of Theorem \ref{thm:local_mild_solution}, it holds that
 \begin{align*}
  \| y(t)\| _Y &\le c \| y_0\|_Y + \|  \Phi_t (yu)\| _Y \le c\|y_0\| _Y + M \|yu \| _{L^2(0,t;Y)} \\
  &= c \| y_0\| _Y + M \mathrm{sup} \left\{ \langle yu,g\rangle_{L^2(0,t;Y)} \ | \ g\in L^2(0,t;Y), \| g\|_{L^2(0,t;Y)} \le 1 \right\}.
 \end{align*}
Hence, for $\varepsilon > 0$ there exists $\tilde{g}\in L^2(0,t;Y), \|\tilde{g}\|_{L^2(0,t;Y)}\le 1,$ such that
\begin{align*}
 \| y(t)\| _Y &\le c \| y_0 \| _Y + M \langle yu , \tilde{g}\rangle_{L^2(0,t;Y)} + \varepsilon .
 \end{align*}
 In particular, for $\varepsilon=\max(\|y_0\|_Y,\|u\|_{L^2(0,T)})$ and $\tilde{c}\!=\!c\|y_0\|_Y+\max(\|y_0\|_Y,\|u\|_{L^2(0,T)})$:
 \begin{align*}
\| y(t)\| _Y &\le \tilde{c} + M \int_0^t \langle y(s), u(s)\tilde{g}(s)\rangle_Y \, \mathrm{d}s  \le \tilde{c} + M \int_0^t \| y(s) \|_Y \| u(s)\tilde{g}(s)\|_Y \, \mathrm{d}s.
\end{align*}
Since $u\tilde{g}\in L^1(0,t;Y)$, Gronwall's lemma implies that
\begin{align}\label{eq:linf_est}
 \| y(t)\| _Y \le \tilde{c} e^{\int_0^t M \| u(s)\tilde{g}(s)\|_Y\, \mathrm{d}s}\le \tilde{c} e^{M\| u\|_{L^2(0,t)}}.
\end{align}
Thus, the solution can be extended beyond $a$ in case $a<T$.
\end{proof}

 \bibliographystyle{siam}
 \bibliography{M149969_refs}
\end{document}